\documentclass{article}
\usepackage[left=4cm, right=4cm, top=4cm, bottom=2.25cm]{geometry}
\usepackage{amsmath,color}
\usepackage{amsfonts,amsthm}
\usepackage{amssymb,enumerate,enumitem,verbatim}
\usepackage[pdftex]{graphicx}
\usepackage[hidelinks]{hyperref}
\usepackage{amsmath,setspace,scalefnt}
\usepackage[usenames,dvipsnames,svgnames,table]{xcolor}
\usepackage{amsfonts,amsthm}
\usepackage{amssymb,enumitem,verbatim}
\usepackage{graphicx}
\usepackage{tikz,caption,subcaption,yfonts}
\usepackage{thmtools}
\usepackage{thm-restate}

\newcounter{capitalcounter}
\newcounter{claimcounter}

\newcounter{myfootnote}[page]

\newtheorem{lemma}{Lemma}[section]

\newtheorem{theorem}[lemma]{Theorem}
\theoremstyle{definition}
\newtheorem{defn}[lemma]{Definition}
\newtheorem{claim}{Claim}[subsection]
\newtheorem{prop}[lemma]{Proposition}

\global\long\def\lem{\begin{lemma}}
\global\long\def\ma{\end{lemma}}
\global\long\def\de{\begin{defn}}
\global\long\def\fn{\end{defn}}

\global\long\def\E{\mathbb{E}}
\global\long\def\P{\mathbb{P}}
\global\long\def\eps{\varepsilon}
\global\long\def\N{\mathbb{N}}
\global\long\def\e{\varepsilon}

\newcommand\subsetsim{\mathrel{%
  \ooalign{\raise0.2ex\hbox{$\subset$}\cr\hidewidth\raise-0.8ex\hbox{\scalebox{0.9}{$\sim$}}\hidewidth\cr}}}

\newcommand\myphead[1]{\medskip

\noindent\textbf{#1}}

\newcommand\cP{\mathcal{P}}
\newcommand\cH{\mathcal{H}}

\newif\ifproofdone
\proofdonefalse
\newcounter{propcounter}

\renewcommand\vec[1]{\overrightarrow{#1}}

\title{Spanning cycles in random directed graphs}

\author{Richard Montgomery\footnote{Mathematics Institute, University of Warwick, Coventry, CV4 7AL, UK.
richard.montgomery@warwick.ac.uk. Research supported by the European Research Council (ERC) under the European Union's Horizon
2020 research and innovation programme (grant agreement no.\ 947978), and the Leverhulme Trust.}}

\begin{document}

\maketitle

\begin{abstract}
We show that, in almost every $n$-vertex random directed graph process, a copy of every possible $n$-vertex oriented cycle will appear strictly before a directed Hamilton cycle does, except of course for the directed cycle itself.
Furthermore, given an arbitrary $n$-vertex oriented cycle, we determine the sharp threshold for its appearance in the binomial random directed graph. These results confirm, in a strong form, a conjecture of Ferber and Long.



\end{abstract}


\section{Introduction}
Hamilton cycles in random graphs have been extensively studied since the early work of Erd\H{o}s and R\'enyi~\cite{ER59} on random graphs. Improving on seminal work by P\'osa~\cite{posa76} and Korshunov~\cite{kor76}, the sharp appearance threshold of the Hamilton cycle was determined in 1983 by Bollob\'as~\cite{bollo83}, and Koml\'os and Szemer\'edi~\cite{KSham}, who showed that, if $p=(\log n+\log\log n+\omega(1))/n$, then the binomial random graph $G(n,p)$ is, with high probability, Hamiltonian. If $p=(\log n+\log\log n-\omega(1))/n$, then, with high probability, $G(n,p)$ has a vertex with degree at most 1, and therefore contains no Hamilton cycle. For such ranges of $p$, then, with high probability, the property of Hamiltonicity in $G(n,p)$ is exactly concurrent with the property of the minimum degree being at least 2.

Such a result can be made more precise by considering the $n$-vertex \emph{random graph process} $G_0$, $G_1$, \dots, $G_{n(n-1)/2}$, where $G_0$ is the graph with vertex set $[n]$ and no edges, and each graph $G_i$, $i\geq 1$, in the sequence, is formed from $G_{i-1}$ by adding a new edge taken uniformly at random from the non-edges of $G_{i-1}$. Independently, Bollob\'as~\cite{bollo84}, and  Ajtai, Koml\'os and Szemer\'edi~\cite{AKS85}, showed that, in almost every random graph process, the first graph $G_i$ with minimum degree at least 2 is Hamiltonian. Further results on the Hamiltonicity of random graphs, including counting and packing results, can be found in Frieze's comprehensive bibliography~\cite{friezebib}.

Hamilton cycles have also been extensively studied in random directed graphs (digraphs). Here, a directed Hamilton cycle is a cycle through every vertex of a digraph whose edges are directed in the same direction around the cycle. In 1980, McDiarmid~\cite{McD83} gave a beautiful coupling argument which, when applied to Hamilton cycles, shows that, if $p=(\log n+\log\log n+\omega(1))/n$, then $D(n,p)$ is Hamiltonian with high probability, where $D(n,p)$ is  the binomial random digraph with $n$ vertices and edge probability $p$. This coupling argument is crucial to this paper, and is covered in Sections~\ref{sec:guide} and~\ref{subsec:mcd}. For \emph{directed} Hamiltonicity, the natural local obstruction is that each vertex must have at least one in-neighbour and at least one out-neighbour so that a directed cycle may pass through it. In $D(n,p)$, if $p=(\log n+\omega(1))/n$, then, with high probability, each vertex will have this property, while, if $p=(\log n-\omega(1))/n$, with high probability at least one vertex will not. Similarly to the undirected case, this local obstruction coincides with when we can expect the binomial random digraph to be Hamiltonian, as shown by Frieze~\cite{frieze88}. That is, if $p=(\log n+\omega(1))/n$, then, with high probability, $D(n,p)$ is Hamiltonian.

The $n$-vertex random digraph process $D_0,D_1,\ldots,D_{n(n-1)}$ begins with the digraph $D_0$ with vertex set $[n]$ and no edges, and each digraph $D_i$, $i\geq 1$, in the sequence is formed from $D_{i-1}$ by adding a new directed edge taken uniformly at random from the non-edges of $D_{i-1}$. Frieze~\cite{frieze88} gave the corresponding result for Hamilton cycles in the random digraph process to that shown in the random graph process. That is, in almost every random digraph process, the first digraph in which every vertex has in- and out-degree at least 1 is Hamiltonian.

The directed $n$-vertex cycle is the most natural generalisation of the undirected $n$-vertex cycle, but we may also consider other $n$-vertex oriented cycles. An oriented cycle is any digraph formed by taking an undirected cycle and orienting its edges.
Ferber and Long~\cite{FL15} studied such cycles in the binomial random digraph, and noted that McDiarmid's coupling argument gives that, for any $n$-vertex oriented cycle $C$,  if $p=(\log n+\log\log n+\omega(1))/n$, then $D(n,p)$ contains a copy of $C$ with high probability. Furthermore, they conjectured that this should be true as long as $p=(\log n+\omega(1))/n$.

The local obstruction to a copy of an $n$-vertex oriented cycle $C$ in $D(n,p)$ is different depending on the pattern of directions on $C$. For example, consider the anti-directed Hamilton cycle, where, for even $n$, the edges change direction at every opportunity around the cycle so that each vertex has in-degree 0 or out-degree 0. If $p=(\log n+2\log\log n+\omega(1))/2n$, then with high probability every vertex in $D(n,p)$ has out-degree at least 2 or in-degree at least 2, and thus has no local obstruction to the containment of an anti-directed cycle. This is tight up the function $\omega(1)$, and, very recently, Frieze, P\'erez-Gim\'enez and Pra\l at~\cite{FPP} confirmed that this is also when we may expect an anti-directed Hamilton cycle to appear in $D(n,p)$.
Thus, compared to the directed Hamilton cycle, we need only around one half of the edge probability to typically find a anti-directed Hamilton cycle. More generally, Frieze, P\'erez-Gim\'enez and Pra\l at~\cite{FPP} studied $n$-vertex cycles in which the pattern of edges repeats after a fixed interval (with respect to $n$), and determined which local conditions are likely to imply the existence of such a cycle in the random digraph process. Indeed, they showed that, except for the anti-directed and directed Hamilton cycle, these cycles are likely to appear in the random digraph process as soon as each vertex has total in- and out-degree at least 2.

In this paper, we show that, with high probability, a much larger range of $n$-vertex oriented cycles will appear in the random digraph process as soon as each vertex has total in- and out-degree at least 2. Our condition on the cycle is only that it has at least $n^{1/2+o(1)}$ vertices where the direction of the edges changes (that is, vertices which have in- or out-degree 0) and at least polylogarithmically many vertices where the direction of the edges is maintained (that is, which have in- and out-degree 1). 

Notably, we show that these cycles are likely to appear \emph{simultaneously} at this point in the random digraph process. Determining the threshold for the simultaneous containment of every possible $n$-vertex oriented cycle in $D(n,p)$ was the original motivation behind this work. For this, we show that, in almost every random digraph process, the first digraph in which every vertex has both in- and out-degree at least 1 contains a copy of every $n$-vertex oriented cycle. In fact, the directed Hamilton cycle is likely to be strictly the last such cycle to appear. From these results, it follows simply that, if $p=(\log n+\omega(1))/n$, then $D(n,p)$ contains a copy of every $n$-vertex oriented cycle.
In particular, this confirms the conjecture of Ferber and Long~\cite{FL15} stated above.
These results are summarised in the following theorem.

\begin{restatable}{theorem}{hittime}\label{hittime}
Let $D_0,D_1,\ldots,D_{n(n-1)}$ be the $n$-vertex random digraph process. Let $m_1$ be the largest integer $m$ for which $\delta^+(D_m)=0$ or $\delta^-(D_m)=0$. Then, with high probability,
\begin{enumerate}[label = (\roman{enumi})]
\item $D_{m_1}$ contains a copy of every $n$-vertex oriented cycle except for the directed $n$-vertex cycle, and
\item $D_{m_1+1}$ contains a copy of every $n$-vertex oriented cycle.
\end{enumerate}
Let $m_0$ be the smallest integer $m$ for which $d^+_{D_m}(v)+d^-_{D_m}(v)\geq 2$ for every $v\in V(D_m)$. Then, with high probability,
\begin{enumerate}[label = (\roman{enumi})]\addtocounter{enumi}{2}
\item $D_{m_0}$ contains a copy of every $n$-vertex oriented cycle with at least $n^{1/2}\log^3n$ changes of direction and at most $n-\log^4n$ changes of direction.
\end{enumerate}
\end{restatable}

We will also find, given any $n$-vertex oriented cycle $C$, the sharp threshold for the appearance of $C$ in $D(n,p)$, where the thresholds vary from $p=\log n/2n$ to $p=\log n/n$ (see Theorem~\ref{sharpthres}). If $C$ has few vertices with in- and out- degree 1, then the random graph must have few vertices with both in- and out-degree exactly 1. If $C$ has few vertices with in-degree 0 or out-degree 0, then the random graph must have few vertices with in-degree 0 or out-degree 0.
As $p$ increases from $(1+o(1))\log n/2n$ to $(1+o(1))\log n/n$, the expected number of vertices in $D$ with in- or out-degree 0 decreases from $n^{1/2+o(1)}$ to 0. The expected number of vertices with both in- and out-degree 1 is much smaller, and, as $p$ increases in this interval, it quickly decreases from $n^{o(1)}$ to 0. Thus, for the sharp threshold for $C$ we focus on the vertices in $C$ with in- or out-degree 0. We define $p_C$ below, before showing that this is the sharp threshold in Theorem~\ref{sharpthres}.

\begin{defn}
Given an $n$-vertex oriented cycle $C$, let $\lambda(C)$ be the number of vertices of $C$ with in- or out-degree 0 in $C$. If $\lambda(C)=0$, then let $p_C=\log n/n$, while if $\lambda(C)>0$, let $p_C=\max\{\log n,2(\log n-\log \lambda(C))\}/2n$.
\end{defn}

\begin{restatable}{theorem}{sharpthres}\label{sharpthres}
For each $\eps>0$ and function $p=p(n)$, with high probability, $D(n,p)$ contains a copy of every $n$-vertex oriented cycle with $p_C\leq (1-\eps)p$ and no copy of any $n$-vertex oriented cycle with $p_C\geq (1+\eps)p$.
\end{restatable}

Both Theorem~\ref{hittime} and Theorem~\ref{sharpthres} follow from a stronger theorem, Theorem~\ref{thmrandprocess}, which gives a better indication of where in the random digraph process we can expect an arbitrary spanning oriented cycle to appear.
However, there are $n$-vertex cycles $C$ whose point of appearance cannot be (with high probability) determined only from the evolving degree sequence of the $n$-vertex random digraph process. For example, consider an $n$-vertex cycle $C$ with exactly two vertices with out-degree 0 and exactly two vertices with in-degree 0, which are in sequence $\ell$ vertices apart on the cycle, for some function $\ell=\ell(n)$. With positive probability the last two vertices in the $n$-vertex random digraph process $D_0,D_1,\ldots,D_{n(n-1)}$ with in- or out-degree 0 will both have in-degree 0. Then, whether a copy of $C$ appears in the first digraph $D_j$ which has at most 2 vertices with out-degree 0 and at most 2 vertices with in-degree 0 can depend, for certain values of $\ell\approx \log n/2\log\log n$ on the different paths in $D_j$ with length $2\ell$ between the last two vertices with in-degree 0. Carefully selecting the value of $\ell$, we can find a sequence $\ell=\ell(n)$ where the probability a copy of $C$ exists in $D_j$ is bounded away from 0 and 1.

The key new method used by this paper is a combination of constructive techniques along with McDiarmid's coupling. After stating our notation, this is sketched in detail in Section 2, before we state our main technical theorem, Theorem~\ref{mainthm}, and its application to the random digraph process, Theorem~\ref{thmrandprocess}. In Section~\ref{sec:provemainthm}, we prove Theorem~\ref{mainthm}, from which we then deduce Theorem~\ref{thmrandprocess} in Section~\ref{sec:di}.

\section{Preliminaries}

\subsection{Notation}
A digraph $D$ has vertex set $V(D)$ and edge set $E(D)$, where $E(D)$ is a collection of ordered distinct vertex pairs from $V(D)$. We let $e(D)=|E(D)|$ and $|D|=|V(D)|$. We say that $uv$ is an edge of $D$ if $(u,v)\in E(D)$, and consider this edge as directed from $u$ to $v$. Where $uv\in E(D)$, we say that $v$ is an out-neighbour (or $+$-neighbour) of $u$ and $u$ is an in-neighbour (or $-$-neighbour) of $v$. For each $\diamond\in \{+,-\}$, we let $N^\diamond_D(v)$ be the set of $\diamond$-neighbours of $v$ in $D$, and set $d^\diamond_D(v)=|N^\diamond_D(v)|$. Where it is clear from context, we omit the subscript. We let $\Delta^+(D)$ and $\delta^+(D)$ be the maximum and minimum out-degree of $D$ respectively, and define $\Delta^-(D)$ and $\delta^-(D)$ similarly. We use notation like $\bar{D}$ to denote a digraph related to $D$ (defining it precisely each time), but never consider complements of graphs or digraphs or use this to denote them.
Where $\pm$ is used, we mean that the statement holds with $\pm$ replaced by both $+$ and by $-$. The length of a path is how many edges it has.

Given $A,B\subset V(D)$, $v\in V(D)$ and $\diamond\in \{+,-\}$, we let $N^\diamond_D(v,B)=N^\diamond_D(v)\cap B$, and $N^\diamond_D(A)=(\cup_{w\in A}N^\diamond_D(w))\setminus A$ and $N^\diamond_D(A,B)=N^\diamond_D(A)\cap B$. The digraph $D[A]$ is the digraph $D$ induced on the vertex set $A$. Given an edge $e$ with vertices in $V(D)$, the digraphs $D+e$ and $D-e$ have vertex set $V(D)$ and edge sets $E(D)\cup \{e\}$ and $E(D)\setminus \{e\}$ respectively. Given a vertex set $A\subset V(D)$, the digraph $D-A$ is the digraph $D[V(D)\setminus A]$. We use similar notation for, for example, $D-v$ with $v\in V(D)$, and $D+E$, where $E$ is a set of edges.

Given two digraphs $H$ and $G$, we say that $H\subsetsim G$ if $G$ contains a copy of $H$. We denote by $D(n,p)$ the binomial random digraph with vertex set $[n]=\{1,\ldots,n\}$, in which each possible edge $uv$ is included independently at random with probability $p$. In a digraph $D$, we say vertices $u,v\in V(D)$ are at least $k$ apart if their graph distance in the underlying undirected graph of $D$ is at least $k$, where the underlying graph of a digraph $D$ has vertex set $V(D)$ and edge set $\{uv:\vec{uv}\text{ or }\vec{vu}\in E(D)\}$.

Our notation for graphs is analogous to that defined above for digraphs. We also use standard asymptotic notation such as $O(n)$, $\Theta_C(n)$, where the implicit constant(s) depend on the variable in the subscript, if any. We say an event holds \emph{with high probability} if the probability which with it holds tends to 1 as $n\to \infty$. For each $n\in \N$, we use $\log^{[2]}n=\log\log n$ and $\log^{[3]}n=\log\log\log n$. All logarithms are natural, and we omit rounding signs whenever they are not crucial.

\subsection{Proof overview}\label{sec:guide}
In our proof sketch we will concentrate on how to show that many different spanning cycles appear simultaneously in the binomial random digraph.
Let us say then that $p=\lambda\log n/n$, for some large constant $\lambda$, and that we wish to show that $D(n,p)$ contains a copy of every $n$-vertex oriented cycle with high probability. We note first that a simple union bound is not strong enough for this. Indeed, given an $n$-vertex cycle $C$ whose edges are oriented with any directions, we have, for $D=D(n,p)$ and an arbitrary $v\in [n]$,
that
\[
\P(C\subsetsim D(n,p))\leq \P(d_D^+(v)+d_D^-(v)>0)= 1-(1-p)^{2(n-1)}=1-\exp(-\Theta_\lambda(\log n)).
\]
As there are $2^{(1-o(1))n}$ oriented cycles with length $n$ up to isomorphism, we thus cannot prove a bound on $\P(C\subsetsim D(n,p))$ before taking a union bound over all the $n$-vertex oriented cycles $C$.

On the other hand, it would suffice to find some pseudorandom properties which $D(n,p)$ has with high probability and show that any digraph with these properties contains any $n$-vertex oriented cycle $C$. However, for general cycles $C$, this seems to be rather challenging.
Instead, we combine the `union bound' approach and the `pseudorandom' approach, as follows.

Taking two random digraphs $D_0$ and $D_1$, each distributed as $D(n,p/2)$, we define a notion of a `pseudorandom digraph' and show that
\begin{equation*}\label{dpseud}
\P(D_0\text{ is pseudorandom})=1-o(1)
\end{equation*}
and, for any $n$-vertex oriented cycle $C$,
\begin{equation}\label{depprob}
\P(C\subsetsim D_0\cup D_1|D_0\text{ is pseudorandom})=1-\exp(-\omega(n)).
\end{equation}
Choosing first the random digraph $D_0$, and then taking a union bound over all cycles $C$, these statements easily combine (see Section~\ref{sec:randpf}) to show that
\[
\P(D(n,p)\text{ contains every $n$-vertex cycle})\geq \P(D_0\cup D_1\text{ contains every $n$-vertex cycle})=1-o(1).
\]

Instead of proving \eqref{depprob} directly, we now employ McDiarmid's coupling technique (as discussed extensively in Section~\ref{subsec:mcd}). For this, consider the following random digraph, $D^\ast(n,q)$.
\begin{defn}
    Let $D^*(n,q)$ be the random digraph with vertex set $[n]$ where each pair of edges ${uv}$ and ${vu}$ are included together independently at random with probability $q$, and otherwise excluded.
\end{defn}
Let $D_1^\ast$ be distributed as $D^*(n,p/2)$. A simple use of McDiarmid's coupling technique (see Section~\ref{subsec:mcd}) shows that
\[
\P(C\subsetsim D_0\cup D_1|D_0\text{ is pseudorandom})\geq \P(C\subsetsim D_0\cup D^\ast_1|D_0\text{ is pseudorandom}).
\]
Therefore, to prove~\eqref{depprob}, it is sufficient to show that
\begin{equation}\label{depprob2}
\P(C\subsetsim D_0\cup D^\ast_1|D_0\text{ is pseudorandom})=1-\exp(-\omega(n)).
\end{equation}

Now, we observe that $D_1^\ast$ has the same distribution as the random graph $G(n,p/2)$ with each edge $uv$ replaced by the two directed edges ${uv}$ and ${vu}$. This allows us to use (undirected) graph techniques to find paths and cycles in the (undirected) underlying graph of $D_1^\ast$, safe in the knowledge that such a path or cycle will exist in $D_1^\ast$ with any orientations on its edges. This is the benefit of working with $D_1^\ast$. However, for~\eqref{depprob2}, we can only use graph techniques which use properties of $G(n,p/2)$ which hold with probability $1-\exp(-\omega(n))$.

Key here is that one part of the standard proof of the Hamiltonicity of $G(n,p)$ uses a `sprinkling' technique that works with probability $1-\exp(-\omega(n))$. So that we may recall this, let $G_0$ and $G_1$ be independent random graphs, each distributed as $G(n,p/2)$. Typically, following the approach pioneered by P\'osa~\cite{posa76}, we show that $G_0$ is likely to be an `expander' (see Section~\ref{sec:comp}), and then, given that $G_0$ is an `expander', that $G_0\cup G_1$ is likely to contain a Hamilton cycle. In fact, we have, where $C_n$ is an $n$-vertex cycle, that
\begin{equation}\label{posaprob}
\P(C_n\subsetsim G_0\cup G_1|G_0\text{ is an `expander'})=1-\exp(-\omega(n)).
\end{equation}
This is the undirected version of \eqref{depprob2}, our version of which we prove as Lemma~\ref{lem-posa}.

Unfortunately, the proof of \eqref{posaprob} (using the extension-rotation method as given in Section~\ref{subsec:posa}) cannot be applied to directed graphs to get an arbitrary cycle. Instead, given a pseudorandom digraph $D_0$, we split $D_1^\ast$ into two random digraphs $D_2^\ast$ and $D_3^\ast$ with equal edge probability and proceed with the following 4 steps.

\begin{enumerate}[label = {\bfseries \Alph{enumi}}]
\item We reveal $D^\ast_2$ to (with very high probability) identify a set of `bad' vertices $B\subset [n]$ which are hard to cover by paths or cycles in $D^\ast_2$.\label{stepA}
\item We use $D_0$ to find sections of the cycle $C$ covering these bad vertices, where the sections have endvertices in $[n]\setminus B$.\label{stepB}
\item We connect these sections using $D^\ast_2$ into a single section of the cycle (using that the endvertices are not `bad'). Say the path found in $D_0\cup D^\ast_2$ is a path $P$ with endvertices $x$ and $y$. When we do this, we ensure that $D^\ast_2-V(P-x-y)$ is an `expander'.\label{stepC}
\item Using our version of \eqref{posaprob}, we reveal $D^\ast_3$ and show that, given $D^\ast_2-V(P-x-y)$ is an `expander', an $x,y$-path through every vertex in $(D^\ast_2\cup D^\ast_3)-V(P-x-y)$ exists with very high probability.\label{stepD}
\end{enumerate}

These steps are given in more detail in Section~\ref{sec:provemainthm}. We next give our definition of pseudorandomness and the statement of our main technical theorem, before discussing how it can be applied to the random digraph process.

\subsection{Pseudorandomness and our main technical theorem}\label{subsec:pseud}
For our definition of pseudorandomness, we take the simplest conditions we need for our methods to work. We require our pseudorandom digraph to satisfy some maximum in- and out-degree condition (\ref{pseudo0} below), some minimum in- and out-degree condition (\ref{pseudo1} below), and a condition that gives rise to some digraph `expansion' (\ref{pseudo2} below). Additionally, the pseudorandom digraph $D$ has an \emph{exceptional set} of vertices $X$, which will arise from low in- or out-degree vertices in the random digraph process.

\begin{defn}\label{def:pseud}
Given an $n$-vertex digraph $D$ and a vertex set $X\subset V(D)$, $D$ is \emph{pseudorandom with exceptional set $X$} if the following hold.
\stepcounter{propcounter}
\begin{enumerate}[label = \textbf{\Alph{propcounter}\arabic{enumi}}]
\item $\Delta^\pm(D)\leq 100\log n$.\label{pseudo0}
\item For each $v\in V(D)$, $d^\pm(v,V(D)\setminus X)\geq \log n/500$.\label{pseudo1}
\item For any sets $A,B\subset V(D)$ and $\diamond\in \{-,+\}$, with $|A|\leq n\log\log n/\log n$, and such that, for each $v\in A$, $d^\diamond(v,B)\geq (\log n)^{2/3}$, we have $|B|\geq |A|(\log n)^{1/3}$.\label{pseudo2}
\end{enumerate}
\end{defn}

We wish to apply our methods to digraphs in the $n$-vertex random digraph process which have minimum in- or out-degree strictly less than $\log n/500$. As \ref{pseudo1} will not hold, such a digraph is not pseudorandom. However, as discussed below, we will modify such a digraph into a pseudorandom digraph. Once we find a spanning cycle, we will undo this modification, and therefore the spanning cycle we find will need to have some additional properties. This motivates our main theorem, Theorem~\ref{mainthm}, where prespecified vertices in the cycle are copied to prespecified vertices in the exceptional set $X$.

\begin{restatable}{theorem}{maintheorem}\label{mainthm}
There is some $n_0$ such that the following holds for each $n\geq n_0$. Let $D_0$ be an $n$-vertex digraph containing $X\subset V(D_0)$, with $|X|\leq n^{3/4}$, so that $D_0$ is pseudorandom with exceptional set $X$. Let $C$ be an $n$-vertex oriented cycle and let $P\subset C$ be a path with length at most $n/10$. Let $f:X\to V(P)$ be an injection such that the vertices in $f(X)$ are pairwise at least $20\log n/\log\log n$ apart on $P$, and let $D_1=D(n,\log n/(10^3n))$.

Then, with probability at least $1-2\exp(-2n)$, $D_0\cup D_1$ contains a copy of $C$ in which $f(x)$ is copied to $x$ for each $x\in X$.
\end{restatable}

We prove this theorem with the strategy outlined in Section~\ref{sec:guide}. To apply it to the random digraph process, we first modify the random digraphs. This is explained in more detail in Section~\ref{sec:di}, but, roughly, before Theorem~\ref{mainthm} is applied, we first use a conditioning argument (from Krivelevich, Lubetzky and Sudakov~\cite{krivcores}) to reserve some random edges to act as $D_1$ in Theorem~\ref{mainthm}, and then identify the low in- or out-degree vertices in the random digraph process. Assigning each low degree vertex two neighbours, we embed well-spaced subpaths of length 2 from the cycle to these vertices and their chosen neighbours.
We then contract each of these well-spaced subpaths of length 2 in the cycle to a single vertex. Next, we contract each embedded subpath of length 2 in the random digraph process and adjust the edges in and out of each contracted vertex so that, if a copy of the contracted cycle is found in the contracted digraph with the contracted vertices in the cycle mapped appropriately to the contracted vertices in the digraph, then undoing the contractions will find a copy of the original cycle in the original digraph (this is described in detail at the start of Section~\ref{sec:randpf}). Thus, we can apply Theorem~\ref{mainthm} using the random edges we have reserved as $D_1$ and with the contracted vertices as the exceptional set $X$. Altogether, this will give us the following theorem.

\begin{restatable}{theorem}{thmrandprocess}\label{thmrandprocess}
Let $D_0,D_1,\ldots,D_{n(n-1)}$ be the $n$-vertex random digraph process. For each $i$ with $0\leq i\leq n(n-1)$, let $s_i$ be the number of vertices in $D_i$ with in-degree or out-degree $0$ and let $t_i$ be the number of vertices in $D_i$ with in-degree 1 and out-degree 1.

Then, with high probability, the following holds for each $i\in \{0,1,\ldots,n(n-1)\}$. If $d^+_{D_i}(v)+d^-_{D_i}(v)\geq 2$ for each $v\in V(D_i)$, then $D_i$ contains a copy of every $n$-vertex oriented cycle with at least $1+(s_i-1)\log n$ changes in direction and at most $n-1-(t_i-1)\log n$ changes in direction.
\end{restatable}

Theorem~\ref{hittime} and Theorem~\ref{sharpthres} follow straight-forwardly from Theorem~\ref{thmrandprocess}. We show this in detail in Section~\ref{sec:randpf}, following the proof of Theorem~\ref{thmrandprocess}. In Section~\ref{sec:provemainthm}, we prove Theorem~\ref{mainthm}.



\section{Proof of Theorem~\ref{mainthm}}\label{sec:provemainthm}
In Section~\ref{sec:comp} we state our component results, before combining them to prove Theorem~\ref{mainthm} in Section~\ref{sec:mainthmproof}. We prove these component results in Sections~\ref{subsec:mcd}--\ref{subsec:posa}.

\subsection{Components of the proof of Theorem~\ref{mainthm}}\label{sec:comp}

The following component parts are contextualised in the proof guide in Section~\ref{sec:guide}.

\medskip

\setcounter{subsubsection}{1}

\noindent{\bfseries \thesubsubsection ~Coupling argument.} We prove the following implication of McDiarmid's coupling argument  in Section~\ref{subsec:mcd}, where an oriented graph is a directed graph in which there is at most one edge between any pair of vertices.

\begin{theorem}\label{thm-couplingarg}
Let $p\in [0,1]$ and $n\in \N$. Let $\cH$ be a set of oriented graphs with vertex set $[n]$, let $D_0$ be a digraph with vertex set $[n]$, let $D_1=D(n,p)$ and let $D^\ast_1=D^\ast(n,p)$.
Then, \[\P(\exists H\in \cH:H\subset D_0\cup D_1)\geq \P(\exists H\in \cH:H\subset D_0\cup D^\ast_1).\]
\end{theorem}

\medskip

\stepcounter{subsubsection}

\noindent{\bfseries \thesubsubsection ~Partitioning lemma.} We partition the vertex set of our pseudorandom digraph for the steps outlined in Section~\ref{sec:guide} using the following lemma, which is proved in Section~\ref{subsec:div}.
\begin{lemma}\label{lem-div} There is some $n_0$ such that the following holds for each $n\geq n_0$.
Suppose $D_0$ is an $n$-vertex pseudorandom digraph with exceptional set $X\subset V(D_0)$ satisfying $|X|\leq n^{3/4}$.

Then, there is a partition $V_0\cup V_1\cup V_2$ of $V(D_0)\setminus X$ with $|V_1|=|V_2|=\lfloor n/4 \rfloor$ such that, for each $v\in V(D_0)$ and $i\in [2]$, we have $d^\pm(v,V_i)\geq \log n/5000$.
\end{lemma}

\medskip

\stepcounter{subsubsection}

\noindent{\bfseries \thesubsubsection ~Steps \ref{stepA} and \ref{stepC}.}
The next lemma, Lemma~\ref{lem-connect}, carries out Steps \ref{stepA} and \ref{stepC} in Section~\ref{sec:guide}, and is proved in Sections~\ref{subsec:connectstart} and~\ref{subsec:connectend}. It identifies the set of `bad' vertices $B$ for Step \ref{stepA}, while showing the connectivity property that is then used for Step \ref{stepC}. We state it after formalising how we need `small sets to expand', as follows.

\begin{defn}
An $n$-vertex graph $G$ is a \emph{10-expander} if it is connected and, given any subset $A\subset V(G)$ with $|A|\leq n/20$, we have $|N(A)|\geq 10|A|$.
\end{defn}

\begin{lemma}\label{lem-connect}  There is some $n_0$ such that the following holds for each $n\geq n_0$.
Let $V_0\subset [n]$ satisfy $|V_0|\geq n/4$, and let $p=\log n/(10^4n)$ and $G=G(n,p)$. Then, with probability at least $1-\exp(-2n)$, $G$ contains a set $B\subset [n]$ with $1\leq |B|\leq n\log^{[3]}n/\log n$ and the following property.

Suppose we have any $k\geq 1$, and any integers $\ell_i\geq 10\log n/\log^{[2]} n$, $i\in [k]$, such that $\sum_{i\in [k]}\ell_i\leq n/8$, and distinct vertices $x_1,\ldots,x_k,y_1,\ldots,y_k\in V(G)\setminus (V_0\cup B)$. Then, there is a set of vertex-disjoint paths $P_1,\ldots, P_{k}$ in $G$ such that the following hold with $V_1=V_0\setminus(B\cup (\cup_{i\in [k]}V(P_i)))$.
\begin{itemize}
\item For each $i\in [k]$, $P_i$ is an $x_{i},y_{i}$-path in $G$ with length $\ell_i$ and internal vertices in $V_0\setminus B$.
\item For each $A\subset V(G)$ with $V_1\subset A$ and $A\cap B=\emptyset$, $G[A]$ is a 10-expander.
\end{itemize}
\end{lemma}

\medskip

\stepcounter{subsubsection}

\noindent{\bfseries \thesubsubsection ~Step \ref{stepB}.}
The next lemma, Lemma~\ref{lem-cover}, carries out Step \ref{stepB}, and is proved in Section~\ref{subsec:cover}. The lemma will be applied to the exceptional set $X$ of the pseudorandom digraph $D_0$ and the set of `bad' vertices $B$ from Lemma~\ref{lem-connect} (split as $B=B^+\cup B^-$). We use it to embed paths from the cycle to cover $X\cup B$, with vertices in $X$ embedded to prespecified vertices, so that these paths have endvertices outside of $X\cup B$ (so that their endvertices are `good' vertices).

\begin{lemma}\label{lem-cover} There is some $n_0$ such that the following holds for each $n\geq n_0$. Let $D$ be an $n$-vertex pseudorandom digraph with exceptional set $X$. Let $B^+,B^-, A^+,A^-$ be disjoint sets in $V(D)\setminus X$, and suppose that
\begin{itemize}
\item for each $v\in V(D)$ and $\diamond\in \{+,-\}$, $d^\diamond(v,B^\diamond\cup A^\diamond)\geq \log n/5000$, and
\item $|X|,|B^+|,|B^-|\leq n\lg{3}/\log n$.
\end{itemize}

Let $B=B^+\cup B^-$. Let $k=|X\cup B|$, and let $\{P_j:j\in [k]\}$ be a collection of vertex-disjoint oriented paths, each with length $2\lceil 4\log n/\log\log n\rceil$. Let $x_j$ be the midpoint of $P_j$ for each $j\in [k]$, and let $f:X\cup B\to [k]$ be a bijection.

Then, there is some $\bar{B}\subset B$ and a collection of vertex-disjoint oriented paths $\{Q_v:v\in X\cup \bar{B}\}$ such that,
\stepcounter{propcounter}
\begin{enumerate}[label = {\bfseries \emph{\Alph{propcounter}\arabic{enumi}}}]
  \item for each $v\in X\cup \bar{B}$, $Q_v$ is a copy in $D$ of a portion of $P_{f(v)}$ with endvertices in $A^+\cup A^-$ and interior vertices in $X\cup B$, in which $x_{f(v)}$ is copied to $v$, and \label{new1}
  \item $X\cup B$ is contained in $\cup_{v\in X\cup \bar{B}}V(Q_v)$.\label{new2}
\end{enumerate}
\end{lemma}

\medskip

\stepcounter{subsubsection}

\noindent{\bfseries \thesubsubsection ~Step \ref{stepD}.}
Finally, the following lemma is used for Step \ref{stepD} and is proved in Section~\ref{subsec:posa}.

\begin{lemma}\label{lem-posa} There is some $n_0$ such that the following holds for each $n\geq n_0$.
Let $G_0$ be a 10-expander with vertex set $[n]$ and let $x,y\in V(G_0)$ be distinct. Let $p=\log n/(10^5n)$ and $G_1=G(n,p)$.
Then, with probability at least $1-\exp(-4n)$, $G_0\cup G_1$ contains a Hamilton $x,y$-path.
\end{lemma}

\subsection{Proof of Theorem~\ref{mainthm}}\label{sec:mainthmproof}
We now put these component parts together to prove Theorem~\ref{mainthm}, as follows.

\ifproofdone

\bigskip

\noindent\textbf{Proof of Theorem~\ref{mainthm} from the component parts is complete.}
\else

\begin{proof}[Proof of Theorem~\ref{mainthm}] Let $n_0$ be sufficiently large that each property in Lemmas~\ref{lem-div},~\ref{lem-connect},~\ref{lem-cover} and~\ref{lem-posa} holds for each $n\geq n_0/2$, and further simple inequalities involving $n\geq n_0$ hold as used below. We will show the property in Theorem~\ref{mainthm} holds for each $n\geq n_0$. For this, let $D_0$ be an $n$-vertex pseudorandom digraph with exceptional set $X\subset V(D_0)=[n]$
such that $|X|\leq n^{3/4}$. Let $C$ be an $n$-vertex oriented cycle. Let $P\subset C$ be a path with length at most $n/100$. Let $f:X\to V(P)$ be an injection, so that the vertices in $f(X)$ are pairwise at least $100\log n/\log\log n$ apart on $P$.

Let $\cH$ be the set of copies of $C$ with vertex set $[n]$ in which $f(x)$ is copied to $x$ for each $x\in X$, and let $p=\log n/(10^4n)$. Independently, let $D^\ast_1$, $D^\ast_2$, $D_1$, and $D^\ast$ be distributed as $D^\ast(n,p)$, $D^\ast(n,p)$, $D(n,10p)$, and $D^*(n,10p)$ respectively. Noting that each pair of edges $\{{uv},{vu}\}$ appears in $D^\ast_1\cup D^\ast_2$ independently at random with probability $1-(1-p)^2\leq 10p$,
Theorem~\ref{thm-couplingarg} implies that
\[
\P(\exists H\in \cH:H\subset D_0\cup D_1)\geq \P(\exists H\in \cH:H\subset D_0\cup D^\ast)\geq \P(\exists H\in \cH:H\subset D_0\cup D^\ast_1\cup D^\ast_2).
\]
Thus, to prove Theorem~\ref{mainthm}, it is sufficient to show that $\P(\exists H\in \cH:H\subset D_0\cup D^\ast_1\cup D^\ast_2)\geq 1-2\exp(-2n)$.

First, using the property from Lemma~\ref{lem-div}, find a partition $V(D_0)\setminus X=V_0\cup V_1\cup V_2$ with $|V_1|=|V_2|=\lfloor n/4\rfloor$ such that, for each $v\in V(D_0)$ and $i\in [2]$, we have $d^\pm(v,V_i)\geq \log n/5000$. Note that $|V_0|\geq n/2-|X|\geq n/4$.

Let $G_1$ be the underlying undirected graph of $D^*_1$, noting that $G_1$ has the same distribution as $G(n,p)$. Due to the choice of $n_0$, we have the property (originating from Lemma~\ref{lem-connect}) that, with probability at least $1-\exp(-2n)$, there exists a set $B\subset V(G_1)$ with $1\leq |B|\leq n\log^{[3]}n/\log n$ such that the following holds.
\stepcounter{propcounter}
\begin{enumerate}[label = \textbf{\Alph{propcounter}\arabic{enumi}}]
\item For any $k\geq 1$, and any integers $\ell_i\geq 10\log n/\log^{[2]}n$, $i\in [k]$, such that $\sum_{i\in [k]}\ell_i\leq n/8$, and distinct vertices $x_1,\ldots,x_k,y_1,\ldots,y_k\in V(G_1)\setminus (V_0\cup B)$, there is a set of internally vertex-disjoint paths
$R_1,\ldots, R_{k}$ such that the following hold with $\bar{V}_0=V_0
\setminus (B\cup (\cup_{i\in [k]}V(R_i)))$.\label{connectprop}
\begin{itemize}
\item For each $i\in [k]$, $R_i$ is an $x_{i},y_{i}$-path in $G_1$ with length $\ell_i$ and internal vertices in $V_0$.
\item For each $A\subset V(G_1)$ with $\bar{V}_0\subset A$ and $A\cap B=\emptyset$, $G_1[A]$ is a 10-expander.
\end{itemize}
\end{enumerate}
Now, let $k'=|X\cup B|\geq 1$ and note that $k'\leq 2n\log^{[3]}n/\log n$. Let $A^+=V_1\setminus B$, $A^-=V_2\setminus B$, $B^+=B\cap V_1$ and $B^-=B\cap V_2$. Let $\ell=2\lceil 4\log n/\log^{[2]} n\rceil$
 and let $\ell_0=10\log n/\log^{[2]} n$. Let $P'$ be the subpath of $C$ containing $P$ which contains $\ell$ extra vertices on each side.
Take paths $P_i$, $i\in [k']$, in $P'$ which are pairwise a distance at least $\ell_0$ apart in $C$, which each have length $2\ell$, and such that, for each $x\in X$, there is some $j\in [k']$ for which $f(x)$ is the center vertex of $P_{j}$. Note that this is possible as $k'(2\ell+2\ell_0)=o(n)$.

Due to the choice of $n_0$, we have the property (originating from Lemma~\ref{lem-cover}) that there is some $k\leq k'$ and a set of vertex-disjoint paths $Q_i$, $i\in [k]$, in $D_0$, such that the following hold.
\begin{enumerate}[label = \textbf{\Alph{propcounter}\arabic{enumi}}]\stepcounter{enumi}
\item For each $i\in [k]$, $Q_i$ is a copy of a portion of $P_i$ with endvertices in $A^+\cup A^-$ and interior vertices in $X\cup B$.\label{B2}
\item $X\cup B\subset \cup_{i\in [k]}V(Q_i)$.\label{B3}
\item For each $x\in X$, for some $j\in [k]$, $Q_{j}$ is the copy of a portion of $P'$ containing $x$ in which $f(x)$ is copied to $x$.\label{B4}
\end{enumerate}
Note that we can assume, by deleting paths if necessary, that each path $Q_i$, $i\in [k]$, contains some vertex in $X\cup B$, and hence, by \ref{B2}, has length at least 2. As $k'=|X\cup B|\geq 1$, we have $k\geq 1$.

Pick an arbitrary clockwise direction on $C$. Relabelling, if necessary, assume that the paths $P_1,\ldots,P_k$ occur on $C$ in clockwise order. For each $i\in [k-1]$, let $\ell_i\geq \ell_0$ be the length of the path between the preimage of $Q_i$ and the preimage of $Q_{i+1}$ on $P$. For each $i\in [k]$, label the endvertices of $Q_i$ so that $Q_i$ is an $x_i,y_i$-path and the preimage of $x_i$ occurs earlier in $P_i$ than the preimage of $y_i$ under the clockwise order.
Note that, by the choice of the paths $P_i$, $i\in [k]$, we have that $\sum_{i\in [k]}\ell_i\leq |P|+2\ell\leq n/8$ and, for each $i\in [k-1]$, $\ell_i\geq \ell_0=10\log n/\log^{[2]}n$. Furthermore, the vertices $x_1,\ldots,x_k,y_1,\ldots,y_k$ are all distinct as they are endvertices of vertex-disjoint paths with length at least 2, and, by~\ref{B2}, these vertices are all in $A^+\cup A^-=(V_1\cup V_2)\setminus B=V(D_0)\setminus (V_0\cup B)$.

By \ref{connectprop}, we can find paths $R_i$, $i\in [k-1]$, which are internally vertex-disjoint such that, for each $i\in [k-1]$, $R_i$ is a $y_i,x_{i+1}$-path in $G_1$ with length $\ell_i$ and internal vertices in $V_0$, and such that, setting $\bar{V}_0=V_0\setminus (B\cup (\cup_{i\in[k-1]}V(R_i)))$, the following holds.

\begin{enumerate}[label = \textbf{\Alph{propcounter}\arabic{enumi}}]\addtocounter{enumi}{4}
\item For each $A\subset V(G_1)$ with $\bar{V}_0\subset A$ and $A\cap B=\emptyset$, $G_1[A]$ is a 10-expander.\label{propexp}
\end{enumerate}

Now, for each $R_i$, $i\in [k-1]$, let $R'_i$ be the digraph formed by replacing each edge $uv$ of $R_i$ by both ${uv}$ and ${vu}$. Observe that $(\cup_{i\in [k]}{Q_i})\cup (\cup_{i\in [k-1]}R'_i)\subset D_0\cup D_1^\ast$ contains an oriented $x_1,y_k$-path, $Q$ say, of a portion of $P'$, $P''$ say, in which the following hold.
\begin{itemize}
  \item By \ref{B2}, the paths $Q_i$, $i\in [k]$, have vertices in $A^+\cup A^-\cup (X\cup B)=V_1\cup V_2\cup X\cup B$, and hence, by the definition of $\bar{V}_0$, we have $\bar{V}_0\subset V(G_1)\setminus V(Q)$.
  \item By \ref{B3}, we have $X\cup B\subset V(Q)$, and, hence, by \ref{propexp}, $G_1':=G_1-(V(Q)\setminus \{x_1,y_k\})$ is a 10-expander.
  \item By \ref{B4}, for each $x\in X$, $f(x)\in V(P'')$ is copied to $x\in V(Q)$.
  \end{itemize}

Let $m=|G_1'|$, and note that $m\geq n-|Q|\geq n-|P''|\geq  n/2\geq n_0/2$. Let $G_2$ be the underlying graph of $D^\ast_2[V(G'_1)]$ and note that $G_2$ has the same distribution as $G(m,p)$. Then, due to the choice of $n_0$, we have the property (originating from Lemma~\ref{lem-posa}) that, with probability at least $1-\exp(-4m)$, $G_1'\cup G_2$
contains a Hamilton $y_k,x_1$-path, $S$ say. Let $S^*$ be $S$ with each edge replaced by a directed edge in each direction, and note that, as $S\subset G_1'\cup G_2$, we have that $S^*\subset D^*_1\cup D^*_2$. Finally, note that a copy of $C$ lies in $Q\cup S^\ast\subset D_0\cup D^*_1\cup D^*_2$, in which $f(x)$ is copied to $x$ for each $x\in X$. Therefore, in total, we have found a copy of $C$ in $D_0\cup D^\ast_1\cup D^\ast_2$
in which $f(x)$ is copied to $x$ for each $x\in X$ with probability at least $1-\exp(-2n)-\exp(-4m)\geq 1-2\exp(-2n)$, as required.
\end{proof}
\fi


\subsection{Coupling argument}\label{subsec:mcd}
To recap, for Theorem~\ref{thm-couplingarg}, we have the following situation. We have $p\in [0,1]$, $n\in \N$, and a set $\cH$ of oriented graphs with vertex set $[n]$. We have a digraph $D_0$ with vertex set $[n]$, and random digraphs $D_1=D(n,p)$ and $D^\ast_1=D^\ast(n,p)$. We want to show that
 \[
\P(\exists H\in\cH:H\subset D_0\cup D_1)\geq \P(\exists H\in \cH:H\subset D_0\cup D_1^\ast).
 \]

To do this, we follow closely the approach of McDiarmid~\cite{McD83}. We construct a sequence of random digraphs, denoted $\widehat{D}_0,\widehat{D}_1,\ldots,\widehat{D}_{n(n-1)/2}$, beginning with the random digraph $\widehat{D}_0$ which has same distribution as $D^\ast_1$. Given an arbitrary order of the $n(n-1)/2$ vertex pairs from $[n]$, throughout this sequence of digraphs we steadily decouple each pair of edges ${uv},{vu}$ in $D^\ast_1$ so that more of these edge pairs appear independently of each other.
At the end of this sequence, we have decoupled the appearance of edge pairs in $D^\ast_1$ until the last digraph in the sequence, $\widehat{D}_{n(n-1)/2}$, has the same distribution as $D_1$. Once this sequence of random digraphs is constructed, we show that, as $i$ increases, some digraph in $\cH$ is increasingly likely to appear in $D_0\cup D^\ast_i$ (see Claim~\ref{simpleclaim}, below). The only change from the proof used by McDiarmid~\cite{McD83} is the introduction of a fixed graph, $D_0$, whose edges are added to every random graph in this sequence, but this introduces no additional complication.
\ifproofdone

\smallskip

\noindent\textbf{Proof of Theorem~\ref{thm-couplingarg} is complete.}
\else

\begin{proof}[Proof of Theorem~\ref{thm-couplingarg}]
Let $\ell=n(n-1)/2$ and enumerate $[n]^{(2)}$ as $e_1=\{x_1,y_1\},\ldots,e_\ell=\{x_\ell,y_\ell\}$. For each $0\leq i\leq \ell$, let $X_i$, $Y_i$ and $Z_i$ be independent Bernoulli random variables which are 1 with probability $p$, and 0 otherwise. For each $0\leq j\leq \ell$, let $\widehat{D}_j$ be the random digraph with vertex set $[n]$ and edge set
\[
\{{x_iy_i}:1\leq i\leq j\text{ and }X_i=1\}\cup \{{y_ix_i}:1\leq i\leq j\text{ and }Y_i=1\}\cup \{{x_iy_i},{y_ix_i}: j<i\leq \ell\text{ and }Z_i=1\}.
\]
Note the following.\stepcounter{propcounter}
\begin{enumerate}[label = \textbf{{\Alph{propcounter}\arabic{enumi}}}]
\item $\widehat{D}_0$ has the same distribution as $D^\ast(n,p)$, and hence $D^\ast_1$.\label{bitwarm1}
\item $\widehat{D}_\ell$ has the same distribution as $D(n,p)$, and hence $D_1$. \label{littlewarm}
\item For each $j\in [\ell]$, $E(\widehat{D}_{j-1})\triangle E(\widehat{D}_j)\subset \{{x_jy_j},{y_jx_j}\}$.\label{bitwarm2}
\end{enumerate}
We will show the following claim.

\begin{claim}\label{simpleclaim}
For each $i\in [\ell]$, $\P(\exists H\in \cH: H\subset D_0\cup \widehat{D}_{i})\geq \P(\exists H\in \cH: H\subset D_0\cup \widehat{D}_{i-1})$.
\end{claim}

This is sufficient to prove the lemma. Indeed, it follows from Claim~\ref{simpleclaim} that
\begin{align*}
\P(\exists H\in \cH: H\subset D_0\cup D_1)&\overset{\ref{littlewarm}}{=}\P(\exists H\in \cH: H\subset D_0\cup \widehat{D}_\ell)\geq \P(\exists H\in \cH: H\subset D_0\cup \widehat{D}_{\ell-1})\\
\geq \ldots & \geq \P(\exists H\in \cH: H\subset D_0\cup \widehat{D}_0)\overset{\ref{bitwarm1}}{=} \P(\exists H\in \cH: H\subset D_0\cup D^\ast_1).
\end{align*}
Thus, $\P(\exists H\in \cH: H\subset D_0\cup D_1)\geq \P(\exists H\in \cH: H\subset D_0\cup D^\ast_1)$, as required. It is left then only to prove Claim~\ref{simpleclaim}.

\smallskip

\begin{proof}[Proof of Claim~\ref{simpleclaim}]
Fix an arbitrary $i\in [\ell]$. Let $D_i'=\widehat{D}_i-\{{x_iy_i},{y_ix_i}\}$, so that, by \ref{bitwarm2}, we also have $D_i'=\widehat{D}_{i-1}-\{{x_iy_i},{y_ix_i}\}$. Let $\mathcal{D}$ be the set of possible outcomes for $D_i'$ and fix an arbitrary $D\in \mathcal{D}$. Consider the following three possible cases.
\begin{enumerate}[label = \textbf{\alph{enumi}}]
\item : $D_0\cup D$ contains some $H\in \cH$.\label{case1}
\item : $(D_0\cup D)+\{{x_iy_i},{y_ix_i}\}$ contains no $H\in \cH$.\label{case2}
\item : $D_0\cup D$ contains no $H\in \cH$ but $(D_0\cup D)+\{{x_iy_i},{y_ix_i}\}$ contains some $H\in \cH$.\label{case3}
\end{enumerate}
If case \ref{case1} occurs, then
\[
\P(\exists H\in \cH: H\subset D_0\cup \widehat{D}_{i-1}|D_i'=D)=1=\P(\exists H\in \cH: H\subset  D_0\cup \widehat{D}_{i}|D_i'=D).
\]
If case \ref{case2} occurs, then
\[
\P(\exists H\in \cH: H\subset D_0\cup \widehat{D}_{i-1}|D_i'=D)=0=\P(\exists H\in \cH: H\subset  D_0\cup \widehat{D}_{i}|D_i'=D).
\]
If case \ref{case3} occurs, then, as $\mathcal{H}$ is a set of oriented graphs, at least one of $D_0\cup D+{x_iy_i}$ or $D_0\cup D+{y_ix_i}$ contains some $H\in \mathcal{H}$. Thus,
\[
\P(\exists H\in \cH: H\subset  D_0\cup \widehat{D}_{i}|D_i'=D)\geq p =\P(\exists H\in \cH: H\subset  D_0\cup \widehat{D}_{i-1}|D_i'=D).
\]
Hence, in all cases, and thus for all $D\in\mathcal{D}$, we have
\[
\P(\exists H\in \cH: H \subset D_0\cup \widehat{D}_{i}|D_i'=D)\geq \P(\exists H\in \cH: H \subset D_0\cup \widehat{D}_{i-1}|D_i'=D).
\]
Therefore,
\begin{align*}
\P(\exists H\in \cH: H\subset  D_0\cup \widehat{D}_{i})&=\sum_{D\in \mathcal{D}}\P(D_i'=D)\cdot \P(\exists H\in \cH: H\subset  D_0\cup \widehat{D}_{i}|D_i'=D)
\\
&\geq \sum_{D\in \mathcal{D}}\P(D_i'=D)\cdot \P(\exists H\in \cH: H\subset  D_0\cup \widehat{D}_{i-1}|D_i'=D)\\
&= \P(\exists H\in \cH:H\subset D_0\cup \widehat{D}_{i-1}). 
\end{align*}
This completes the proof of the claim, and hence the lemma.
\renewcommand{\qedsymbol}{$\boxdot$}
\qedhere
\renewcommand{\qedsymbol}{$\square$}
\qedsymbol
\end{proof}
\renewcommand{\qedsymbol}{}
\end{proof}
\renewcommand{\qedsymbol}{$\square$}
\fi


\subsection{Splitting sets with the local lemma}\label{subsec:div}
We will prove Lemma~\ref{lem-div} with a standard application of the following version of the Local Lemma, due to Lov\'asz (see~\cite[Theorem 1.1]{spencer}), where the dependence graph of a set of events $A_1,\ldots,A_n$ is the graph with vertex set $\{A_1,\ldots,A_n\}$ and an edge between $A_i$ and $A_j$, $i,j\in [n]$ and $i\neq j$, exactly when the events $A_i$ and $A_j$ are not independent.

\begin{theorem}\label{genlocal} 
Let $A_1,\ldots,A_n$ be events in a probability space $\Omega$ with dependence graph $K$. Suppose there exist $0<q_1,\ldots,q_n<1$ such that, for each $i\in [n]$,
\[
\P(A_i)\leq q_i\prod_{j:A_iA_j\in E(K)}(1-q_j).
\]
Then,  with strictly positive probability, no such event $A_i$ occurs.
\end{theorem}

We will also use the following well-known Chernoff bound  (see, for example,~\cite[Corollary 2.3]{bollorandomgraphs}).
\lem\label{chernoff} If $X$ is a binomial variable with standard parameters~$n$ and $p$, denoted $X=\mathrm{Bin}(n,p)$, and $\e$ satisfies $0<\e\leq 3/2$, then $\P(|X-\E X|\geq \e \E X)\leq 2\exp\left(-\e^2\E X/3\right)$.
\ma

Using Theorem~\ref{genlocal} and Lemma~\ref{chernoff}, we can now prove Lemma~\ref{lem-div}.

\ifproofdone
\noindent\textbf{Proof of Lemma~\ref{lem-div} is complete.}
\else
\begin{proof}[Proof of Lemma~\ref{lem-div}]
Let $p=1/5$ and $n\geq 10^3$. Let $D$ be a pseudorandom digraph with vertex set $[n]$ and exceptional set $X\subset V(D)$ which satisfies $|X|\leq n^{3/4}$. Take a partition $V(D)\setminus X=U_0\cup U_1\cup U_2$ so that each $v\in V(D)\setminus X$ appears independently in $U_0$, $U_1$ and $U_2$ with probability $1-2p$, $p$ and $p$, respectively.
For each $v\in [n]$, let $E_v$ be the event that, for some $i\in [2]$ and $\diamond\in \{+,-\}$, we have $d^\diamond(v,U_i)< \log n/5000$. Let $E_0$ be the event that either $|U_1|>n/4$ or $|U_2|> n/4$.

If $E_0$ does not hold, then, using that $|X|\leq n^{3/4}\leq n/4$, take a partition $V(D)\setminus X=V_0\cup V_1\cup V_2$ such that $U_1\subset V_1$, $U_2\subset V_2$,  and $|V_1|=|V_2|=\lfloor n/4 \rfloor$. If $E_0$ does hold then let $V_i=U_i$ for each $i\in \{0,1,2\}$. Note that, if $E_0$ does not hold, and no $E_v$, $v\in [n]$, holds, then $V_0\cup V_1\cup V_2$ satisfies the requirements in the lemma. Therefore, it is sufficient to show that there is some $n_0$ such that, whenever $n\geq n_0$, no such event holds with positive probability.

Let $\Delta=100\log n$, so that, by \ref{pseudo0} in the definition of pseudorandomness, $\Delta^\pm(D)\leq \Delta$. Note that, for each $v\in [n]$, each event $E_v$ has some dependence on only $E_0$ and the events $E_u$, $u\in Y_v:=\{u\in [n]:(N^+(u)\cup N^-(u))\cap (N^+(v)\cup N^-(v))\neq\emptyset\}$, and, furthermore, that $|Y_v|\leq 4\Delta^2$.
Let $q_0=1/2$ and $q=\exp(-\sqrt{\log n})$. Thus, the lemma follows from the following claim and Theorem~\ref{genlocal} applied with $q_0$, and $q_v=q$ for each $v\in [n]$.

\begin{claim}\label{breeze} There is some $n_0$ such that, if $n\geq n_0$, then $\P(E_0)\leq q_0(1-q)^n$, and, for each $v\in [n]$, $\P(E_v)\leq q(1-q_0)(1-q)^{4\Delta^2}$.
\end{claim}

\smallskip

\begin{proof}[Proof of Claim~\ref{breeze}] First, take an arbitrary $v\in [n]$. By \ref{pseudo1} we have $d^\pm(v,V(D)\setminus X)\geq \log n/500$. For each $i\in [2]$, then, $\E|N^\pm(v)\cap U_i|\geq \log n/2500$. Therefore, by Lemma~\ref{chernoff} with $\eps=1/2$, we have $\P(E_v)= \exp(-\Omega(\log n))$.
Noting further that $\Delta^2q=o(1)$, we have
\[
\P(E_v)= \exp(-\Omega(\log n))=o(q)=o(q(1-4\Delta^2q))=o(q(1-q_0)(1-q)^{4\Delta^2}).
\]
Thus, for sufficiently large $n$, $\P(E_v)\leq q(1-q_0)(1-q)^{4\Delta^2}$.

Now, as $|[n]\setminus X|\geq n-n^{3/4}\geq 3n/4$, we have $n/10\leq \E|U_i|\leq n/5$ for each $i\in [2]$. Thus, by Lemma~\ref{chernoff} with $\eps=1/100$, we have $\P(E_0)=\exp(-\Omega(n))$. Therefore, for sufficiently large $n$, $q_0(1-q)^n\geq \exp(-2qn)/2=\exp(-o(n))\geq \P(E_0)$, as required. 
\renewcommand{\qedsymbol}{$\boxdot$}
\qedhere
\renewcommand{\qedsymbol}{$\square$}
\qedsymbol
\end{proof}
\renewcommand{\qedsymbol}{}
\end{proof}
\renewcommand{\qedsymbol}{$\square$}
\fi


\subsection{Expansion with very high probability}\label{subsec:connectstart}
In this section, we show how to get an expansion property in some subgraph of a random graph with very high probability, which we then use in Section~\ref{subsec:connectend} to prove Lemma~\ref{lem-connect}.
We start with the following simple proposition concerning the neighbourhoods of large sets in a random graph.

\begin{prop}\label{connsets} For each fixed $c>0$, if $p=c \log n/n$, then, with probability $1-\exp(-\omega(n))$, the following hold in $G=G(n,p)$.
\stepcounter{propcounter}
\begin{enumerate}[label = \textbf{\emph{\Alph{propcounter}\arabic{enumi}}}]
\item Every set $U\subset V(G)$ with $|U|=n\log^{[3]}n/2\log n$ satisfies $|N(U)|\geq 9n/10$.\label{mara1}
\item Every disjoint pair $A,B\subset V(G)$ of subsets of size at least $n/\log^{2/5}n$ have some edge between them.\label{mara2}
\end{enumerate}
\end{prop}
\ifproofdone
\noindent\textbf{Proof of Proposition~\ref{connsets} is complete.}
\else
\begin{proof}
Given any disjoint subsets $U,U'\subset  V(G)$ with $|U|=n\log^{[3]}n/2\log n$ and $|U'|\geq n/100$, the probability there are no edges between $U$ and $U'$ in $G$ is
\[
(1-p)^{|U||U'|}= \exp(-\Omega(p n^2\log^{[3]}n/\log n))=\exp(-\Omega(n\log^{[3]}n)).
\]
Therefore, as there are at most $2^{2n}$ such pairs $U,U'\subset V(G)$, there are no such pairs with no edges between them with probability at least $1-4^{n}\cdot \exp(-\Omega(n\log^{[3]}n))=1-\exp(-\omega(n))$. If such a property holds, then, for any $U\subset  V(G)$ with $|U|=n\log^{[3]}n/2\log n$, we have $|V(G)\setminus (N(U)\cup U)|<n/100$, and hence
\[
|N(U)|\geq n-n/100-|U|\geq 9n/10.
\]
Thus, \emph{\ref{mara1}} holds with probability $1-\exp(-\omega(n))$.

Now, given disjoint subsets $A,B\subset V(G)$ with size at least $n/\log^{2/5}n$, the probability there is no edge between them is
\[
(1-p)^{|A||B|}= \exp(-\Omega(p n^2/\log^{4/5} n)= \exp(- \Omega(n\log^{1/5}n)).
\]
Therefore, as there are at most $2^{2n}$ such pairs $A,B\subset V(G)$, there are no such pairs with no edges between them with probability at least $1-4^{n}\cdot \exp(- \Omega(n\log^{1/5}n))=1-\exp(-\omega(n))$. That is, with probability $1-\exp(-\omega(n))$, \emph{\ref{mara2}} holds.
\end{proof}
\fi

Using \emph{\ref{mara1}} from Proposition~\ref{connsets}, we now find in a random graph a subgraph with almost the same vertex set and a good expansion property, by removing a maximal set $B$ without this expansion property.

\begin{lemma}\label{findexpander}
Let $c>0$, $p=c \log n/n$, $d=\log^{1/3} n$ and $m=n/100d$. Let $V_0\subset [n]$ be a set of at least $n/4$ vertices and $G=G(n,p)$.

Then, with probability $1-\exp(-\omega(n))$, there is a set $B\subset [n]$ with $1\leq |B|\leq n\log^{[3]}n/2\log n$ such that, for each $U\subset V(G)\setminus B$ with $|U|\leq 2m$, we have $|N(U,V_0\setminus B)|\geq d|U|$.
\end{lemma}
\ifproofdone
\noindent\textbf{Proof of Lemma~\ref{findexpander} is complete.}
\else
\begin{proof} By Proposition~\ref{connsets}, we have that \emph{\ref{mara1}} holds in $G$ with probability $1-\exp(-\omega(n))$. Let $B\subset V(G)$ be a largest set satisfying $|B|\leq 3m$ such that $|N(B,V_0)|\leq 2d|B|$, noting that the empty set demonstrates that such a set $B$ exists. Note that, if $B=\emptyset$, then taking an arbitrary set $B'\subset V(G)$ with $|B'|=1$, we have, for each $U\subset V(G)\setminus B'$ with $1\leq |U|\leq 2m$, that $|N(U,V_0\setminus B')|\geq 2d|U|-1\geq d|U|$. Therefore, we can assume that $|B|\geq 1$.
We will now show that $B$ satisfies the property in the lemma, starting with showing that, in fact, $|B|<n\log^{[3]}n/2\log n\leq m$.

As $|B|\leq 3m$ and $|N(B,V_0)|\leq  2d|B|\leq 6dm= 6n/100$, we have
\[
|N(B)|\leq |V(G)\setminus V_0|+|N(B,V_0)|\leq 3n/4+6n/100< 9n/10.
\]
Thus, by \emph{\ref{mara1}}, we have $|B|<n\log^{[3]}n/2\log n\leq m$.

Now, let $U\subset V(G)\setminus B$ with $|U|\leq 2m$ and $U\neq\emptyset$. As $|B\cup U|\leq 3m$, by the choice of $B$ we have that $|N(B\cup U,V_0)|> 2d|B\cup U|$. Then,
\[
|N(U,V_0\setminus B)|\geq |N(B\cup U,V_0)|-|N(B,V_0)|> 2d|B\cup U|-2d|B|=2d|U|\geq d|U|,
\]
as required.
\end{proof}
\fi


\subsection{Expansion into connection}\label{subsec:connectend}
To connect pairs of vertices efficiently with paths using expansion properties we will use the extendability techniques of Glebov, Krivelevich and Johannson~\cite{RomanPhD}. These methods flexibly embed bounded-degree trees in larger graphs using certain expansion conditions, though we will only use them to find paths with specified lengths between specified vertex pairs. (More generally, see \cite[Section 3.1]{myselfspanningtrees} for a practical overview of the use of the $(d,m)$-extendability methods and \cite{myselfdirham} for a directed generalisation for finding consistently oriented paths.)

We first recall the key definition of \emph{$(d,m)$-extendability}, using the following \emph{inclusive neighbourhood} $N'(U)$ of a vertex set $U$.

\de For each $U\subset V(G)$, let $N'(U)=\{u\in V(G):\exists v\in U\text{ s.t.\ }vu\in E(G)\}=\cup_{u\in U}N(u)$.
\fn

\de\label{extendabledefn}
Let $d\geq 3$ and $m\geq 1$, let~$G$ be a graph, and let $S\subset G$ be a subgraph of~$G$. We say that~$S$ is \emph{$(d,m)$-extendable in $G$} if~$S$ has maximum degree at most~$d$ and, for all sets $U\subset V(G)$ with $|U|\leq 2m$,
\begin{equation}\label{extendthing}
|N'(U)\setminus V(S)|\geq(d-1)|U|-\sum_{x\in U\cap V(S)}(d_S(x)-1).
\end{equation}
\fn

Given two vertices in an extendable subgraph, we can add a path with a specified length between them (subject to certain simple conditions) to, crucially, get a subgraph which is still extendable. This allows a sequence of paths to be added while remaining extendable. This is possible using the following lemma.

\begin{lemma}\cite[Corollary 3.12]{myselfspanningtrees}\label{connectcor}
Let $d,m,\ell\in\N$ satisfy $m\geq 1$ and $d\geq 3$. Let $k=\lceil\log (2m)/ \log(d-1)\rceil$ and $\ell\geq 2k+1$. Let~$G$ be a graph in which any two disjoint sets of size $m$ have some edge between them. Let~$S$ be a $(d,m)$-extendable subgraph of~$G$ with at most $|G|-10dm-(\ell-2k-1)$ vertices.

Suppose $a$ and $b$ are two distinct vertices in~$S$, both with degree at most $d/2$ in~$S$. Then, there is an $a$,$b$-path $P$ in $G$, with length $\ell$ and internal vertices outside of~$S$, so that $S+P$ is $(d,m)$-extendable in $G$.
\end{lemma}

We can now prove Lemma~\ref{lem-connect}.

\ifproofdone
\noindent\textbf{Proof of Lemma~\ref{lem-connect} is complete.}
\else
\begin{proof}[Proof of Lemma~\ref{lem-connect}] As in the statement of the lemma, let $V_0\subset [n]$ satisfy $|V_0|\geq n/4$ and let $p=\log n/(10^4n)$. Furthermore, let $d=\log^{1/3} n$ and $m=n/(100d)$.
By Lemma~\ref{findexpander} with $c=10^{-4}$, with probability $1-\exp(-\omega(n))$, $G=G(n,p)$ contains a set $B\subset [n]$ such that the following holds.
\stepcounter{propcounter}
  \begin{enumerate}[label = {\bfseries \Alph{propcounter}\arabic{enumi}}]
    \item $1\leq |B|\leq n\log^{[3]}n/\log n$ and, for any set $U\subset V(G)\setminus B$ with $|U|\leq 2m$, $|N_G(U,V_0\setminus B)|\geq d|U|$.\label{finalprop1}
    \end{enumerate}
    By Proposition~\ref{connsets} with $c=10^{-4}$, with probability $1-\exp(-\omega(n))$, the following holds.
    \begin{enumerate}[label = {\bfseries \Alph{propcounter}\arabic{enumi}}]\stepcounter{enumi}
    \item Any two disjoint sets $U,U'\subset V(G)$ of size at least $m$ have some edge between them in $G$.\label{finalprop2}
  \end{enumerate}
Therefore, there is some $n_0$ such that, for each $n\geq n_0$, with probability at least $1-\exp(-2n)$, $G=G(n,p)$ contains a set $B$ such that \ref{finalprop1} and \ref{finalprop2} hold, and we can assume that $d\geq 12$.

We will now show that $G$ and $B$ have the property in the lemma. Note that, from \ref{finalprop1}, we have the required bounds on $|B|$. Let then $k\geq 1$ and take any integers $\ell_i\geq 10\log n/\log^{[2]}n$, $i\in [k]$, such that $\sum_{i\in [k]}\ell_i\leq n/8$, and distinct vertices $x_1,\ldots,x_k,y_1,\ldots,y_k\in V(G)\setminus (V_0\cup B)$.
Let $S_0$ be the graph with vertex set $V(G)\setminus (V_0\cup B)$ and no edges, and let $G'=G-B$. We will show that $S_0$ is $(d,m)$-extendable in $G'$. First note that $S_0\subset G'$, and $\Delta(S_0)=0$. For each $U\subset V(G')$ with $|U|\leq 2m$, we have
\[
|N_{G'}'(U)\setminus V(S_0)|\geq |N_{G}(U,V_0\setminus B)|\overset{\text{\ref{finalprop1}}}{\geq} d|U|\geq (d-1)|U|-\sum_{x\in U\cap V(S_0)}(d_{S_0}(x)-1),
\]
as required. Thus, $S_0$ is $(d,m)$-extendable in $G'$.

Now, for each $i=1,\ldots,k$ in turn, we can apply Lemma~\ref{connectcor} to find a path $P_i$, such that
\stepcounter{propcounter}
  \begin{enumerate}[label = {\bfseries \Alph{propcounter}\arabic{enumi}}]
\item $P_i$ is an $x_i,y_i$-path in $G'$ with length $\ell_i$ and interior vertices in $V(G')\setminus (\cup_{j=1}^{i-1}V(P_j))$, and \label{firstprop}
\item $S_{i}:=S_0+P_1+\ldots+P_i$ is $(d,m)$-extendable in $G'$.\label{secondprop}
\end{enumerate}
Indeed, suppose that we seek the path $P_i$ for some $i\in [k]$. Then, $S_{i-1}=S_0+P_1+\ldots+P_{i-1}$ is $(d,m)$-extendable in $G'$ by \ref{secondprop} for $i-1$ if $i>1$, or as $S_0$ is $(d,m)$-extendable in $G'$. Now, note that
\begin{align*}
|S_i|&\leq |V(G')\setminus V_0|+\sum_{j=1}^{i-1}\ell_j\leq |G'|- |V_0\setminus B|-\ell_i+\sum_{j\in [k]}\ell_j
\\
&\leq |G'|- \frac{n}{4}+|B|-\ell_i+\sum_{j\in [k]}\ell_j\leq |G'|-\frac{n}{8}+|B|-\ell_i\leq |G'|-10dm-\ell_i.
\end{align*}
Furthermore, $\ell_i\geq 10\log n/\log^{[2]}n\geq 2\lceil \log(2m)/\log(d-1)\rceil +1$. Finally, note that, as $x_1,\ldots,x_k,y_1,\ldots,y_k$ are distinct, by \ref{firstprop}, $x_i$ and $y_i$ have degree $0\leq d/2$ in $S_{i-1}$. Therefore, by Lemma~\ref{connectcor} and \ref{finalprop2}, there is an $x_i,y_i$-path $P_i$ with length $\ell_i$ and interior vertices in $V(G')\setminus V(S_{i-1})=V_0\setminus (B\cup (\cup_{j<i}V(P_j)))$ and such that $S_i:=S_{i-1}+P_i$ is $(d,m)$-extendable. Thus, \ref{firstprop} and \ref{secondprop} are satisfied for $i$.

Suppose then we have paths $P_i$, $i\in [k]$, satisfying \ref{firstprop} and \ref{secondprop}. By \ref{firstprop}, for each $i\in [k]$, $P_i$ is an $x_i,y_i$-path with length $\ell_i$, and the paths $P_i$, $i\in [k]$, are internally vertex-disjoint. Furthermore, for each $i\in [k]$, the internal vertices of $P_i$ are in $V(G')\setminus V(S_{i-1})\subset V(G')\setminus V(S_0)=V_0\setminus B$.

Let $V_1=V_0\setminus (B\cup (\cup_{j\in[k]}V(P_j)) )=V_0\setminus (B\cup V(S_k))$. To show that the paths $P_i$, $i\in [k]$, have the property in the lemma, it is left only to show that, for  any $A\subset V(G)\setminus B$
with $V_1\subset A$ and $A\cap B=\emptyset$, $G[A]$ is a 10-expander. Take then such a set $A$, and let $H:=G[A]\subset G'$.

For each $U\subset V(H)$ with $0<|U|\leq m$, by \ref{secondprop} for $i=k$, as $U\subset V(H)\subset V(G')$ and $V(S_k)=(V(G)\setminus (V_0\cup B))\cup (\cup_{j\in [k]}V(P_j))=V(G')\setminus V_1$, we have
\begin{align}
|N_{H}(U)|&\geq |N'_H(U)|-|U|=|N'_{G'}(U)\cap A|-|U|\geq |N'_{G'}(U,V_1)|-|U|= |N'_{G'}(U)\setminus V(S_k)|-|U|\nonumber\\
&\geq (d-1)|U|-|U|\geq (d-2)|U|\geq 10|U|.\label{thiii}
\end{align}
For each $U\subset V(H)$ with $m<|U|\leq |H|/20$, by \ref{finalprop2} we have that $|V(H)\setminus (U\cup N_H(U))|\leq m$. Therefore, using that $|U|\leq |H|/20$, we have
\[
|N_{H}(U)|\geq |H|-|U|-m\geq |H|-2|U|\geq 10|U|,
\]
so that $|N_H(U)|\geq 10|U|$ for each $U\subset V(H)$ with $|U|\leq |H|/20$.
Finally, by \eqref{thiii}, any connected component of $H$ must contain more than $m$ vertices. Therefore, by \ref{finalprop2}, $H$ has 1 connected component. Thus, $H$ is connected, and hence an $10$-expander, as required.
\end{proof}
\fi


\subsection{Covering vertices with a path}\label{subsec:cover}

We now prove Lemma~\ref{lem-cover}. In a pseudorandom digraph $D$, this allows us to use sections of the cycle $C$ to cover the exceptional set $X$ as well as a further set $B$ of vertices corresponding to the set $B$ found in Lemma~\ref{lem-connect}. To prove Lemma~\ref{lem-cover}, we identify for some $r$ a sequence $B_0\subset B_1\subset B_2\subset \ldots \subset B_r=B$ of subsets of $B$, such that, roughly speaking, the later a vertex first appears in a set in the sequence, the harder it is to cover with paths in $D$.
We then embed paths covering $X$ while using vertices from $B$, before, iteratively, for each $i$ from $r-1$ to $1$, covering the unused vertices from $B\setminus B_i$ while using unused vertices in $B_{i}$ if necessary. This ensures that we cover vertices which are more difficult to cover first.

\ifproofdone

\bigskip

\noindent\textbf{Proof of Lemma~\ref{lem-cover} is complete.}
\else

\begin{proof}[Proof of Lemma~\ref{lem-cover}] Following the lemma statement, let $D$ be an $n$-vertex pseudorandom graph with exceptional set $X$, let $B^-,B^+,A^+,A^-\subset V(D)\setminus X$ be disjoint, such that $X$, $B^+$ and $B^-$ each have size at most $n\log^{[3]}n/\log n$, and such that, for each $v\in V(D)$ and $\diamond\in \{+,-\}$, we have $d^\diamond(v,B^\diamond\cup A^\diamond)\geq \log n/5000$. Let $B=B^+\cup B^-$ and let $P_i$, $i\in [k]$, be vertex-disjoint oriented paths with length $2\lceil 4\log n/\log\log n\rceil$, where $k=|X\cup B|$ and, for each $i\in [k]$, $x_i$ is the midpoint of $P_i$. Let $f:X\cup B\to [k]$ be a bijection.

Let $d=\log n/10^4$ and let $B_0^+=B_0^-=\emptyset$. Iteratively, for each integer $i\geq 1$ and $\diamond\in \{+,-\}$, let
\begin{equation}\label{calm}
B^\diamond_i=\{v\in B^\diamond:d^+(v,B^+_{i-1}\cup A^+)\geq d\;\text{ and }\;d^-(v,B^-_{i-1}\cup A^-)\geq d\}.
\end{equation}
This gives a sequence $B_0^+\subset B_1^+\subset \ldots$ of subsets of $B^+$ and a sequence $B_0^-\subset B_1^-\subset \ldots$ of subsets of $B^-$. For each $i\geq 0$, let $B_i=B_i^+\cup B_i^-$. We will show the following claim.

\begin{claim}\label{decreasingaway}
For each $i\geq 0$, $|B\setminus B_i|\leq |B|/(\log n)^{i/3}$.
\end{claim}
\begin{proof}[Proof of Claim~\ref{decreasingaway}] Note that this is true for $i=0$. We will prove this by induction on $i$, so suppose that $i>0$ and that it is true for $i-1$.
For each $\diamond\in\{+,-\}$, let $Z_i^\diamond$ be the set of vertices $v\in B$ such that $d^\diamond(v,B_{i-1}^\diamond\cup A^\diamond)<d$. Observe that, by \eqref{calm}, $Z_i^+\cup Z_i^-=B\setminus B_i$. Furthermore, for each $\diamond\in \{+,-\}$, every vertex $v\in Z_i^\diamond$ has $d^\diamond(v,B^\diamond\cup A^\diamond)\geq \log n/5000=2d$ but $d^\diamond(v,B_{i-1}^\diamond\cup A^\diamond)<d$. Therefore, for each $v\in Z_i^\diamond$, we have $d^\diamond(v,B^\diamond\setminus B_{i-1}^\diamond)>d\geq (\log n)^{2/3}$. Therefore, by \ref{pseudo2} in the definition of pseudorandomness,
as $|Z_i^\diamond|\leq |B|\leq 2n\log^{[3]}n/\log n$, we have $|B^{\diamond}\setminus B_{i-1}^\diamond|\geq |Z_i^\diamond|(\log n)^{1/3}$.
Therefore,
\begin{align*}
|B\setminus B_i|&=|Z_i^+\cup Z_i^-|\leq |Z_i^+|+|Z_i^-|\leq (|B^+ \setminus B^+_{i-1}|+|B^-\setminus B_{i-1}^-|)/(\log n)^{1/3}\\
&=|B\setminus B_{i-1}|/(\log n)^{1/3}\leq |B|/({\log n})^{i/3}.
\end{align*}
This completes the inductive step, and hence the proof of the claim.
\renewcommand{\qedsymbol}{$\boxdot$}
\end{proof}
\renewcommand{\qedsymbol}{$\square$}

By Claim~\ref{decreasingaway}, if $i\geq 3(\log n/\log\log n)+1$, then $B\setminus B_i=\emptyset$. Let then $r\leq 4\log n/\log\log n$ be the smallest integer such that $B\setminus B_r=\emptyset$, and note that $B^+_{r}=B^+$ and $B^-_{r}=B^-$.
For each $\diamond\in\{+,-\}$, let $H^\diamond$ be the bipartite auxiliary (undirected) graph with vertex classes a copy of $X\cup B$ and a disjoint copy of $B^\diamond\cup A^\diamond$, with an edge $xy$ between $x\in X\cup B$ and $y\in B^\diamond\cup A^\diamond$ if, for some $i\in \{0,1,\ldots,r\}$, $x\in X\cup (B\setminus B_i^\diamond)$ and $y\in B_i^\diamond\cup A^\diamond$, and $y\in N^\diamond_D(x)$.

\begin{claim}\label{forhall} For each $\diamond\in \{+,-\}$ and $U\subset X\cup B$, we have $|N_{H^\diamond}(U)|\geq 2|U|$.
\end{claim}
\begin{proof}[Proof of Claim~\ref{forhall}] Let $\diamond\in \{+,-\}$ and $U\subset X\cup B$. Suppose $x\in U$. If $x\in U\cap X$, then, as $B^\diamond_r=B^\diamond$, for each $y\in (B^\diamond\cup A^\diamond)\cap N^\diamond_D(x)$, $xy\in E(H^\diamond)$. Thus, we have $d_{H^\diamond}(x)\geq \log n/5000>d$.

On the other hand, if $x\in U\cap B$, then let $i$ be the smallest $i\in [r]$ with $x\in B^+_i\cup B^-_i$. Then, by the choice of $B^+_i$ and $B^-_i$, we have $d^\diamond_D(x,B_{i-1}^\diamond\cup A^\diamond)\geq d$.  Thus, we have $d_{H^\diamond}(x)\geq d$.

Therefore, $d_{H^\diamond}(x)\geq d$ for each $x\in U$. Let $V=N_{H^\diamond}(U)$, and let $U'$ and $V'$ be the vertex sets in $D$ which correspond to $U$ and $V$. We have that $d_D^\diamond(x,V')\geq d$ for each $x\in U'$, and therefore, by \ref{pseudo2}, as $|U|\leq |X\cup B|\leq 3n\log^{[3]}n/\log n$, we have that
\[
|N_{H^\diamond}(U)|=|V|=|V'|\geq |U'|(\log n)^{1/3}\geq 2|U'|=2|U|,
\]
as required.
\renewcommand{\qedsymbol}{$\boxdot$}
\end{proof}
\renewcommand{\qedsymbol}{$\square$}

Therefore, for each $\diamond\in\{+,-\}$, the appropriate Hall's matching criterion holds by Claim~\ref{forhall} to show that there exist functions $g^\diamond_1,g^\diamond_2:X\cup B\to B^\diamond\cup A^\diamond$ so that $vg_i(v)\in E(H^\diamond)$ for each $v\in X\cup B$ and $i\in [2]$, and $g^\diamond_1(v),g^\diamond_2(v)$, $v\in X\cup B$, are all distinct vertices in $B^\diamond\cup A^\diamond$.

For each $v\in X\cup B$ we now find a path $Q_v$ covering $v$, in which $f(x)$ is copied to $x$ and then, moving in either direction on $Q_v$, the vertices appear earlier and earlier in the sequence $A^+\cup A^-\cup B_0, B_1,\ldots, B_r$. To do this, for each $v\in X\cup B$, let $Q_v\subset D$ be a longest path satisfying the following properties.
\stepcounter{propcounter}
\begin{enumerate}[label = {\bfseries \Alph{propcounter}\arabic{enumi}}]
  \item $Q_v$ is a copy of a portion of $P_{f(v)}$ containing $x_{f(v)}$ in which $x_{f(v)}$ is copied to $v$.\label{pointone}
  \item Each interior vertex of $Q_v$ is in $X\cup B$, and the endvertices of $Q_v$ are in $X\cup B\cup A^+\cup A^-$.\label{pointtwo}
  \item For each $u\in V(Q_v)$ and $\diamond\in \{+,-\}$, if $w$ is a $\diamond$-neighbour of $u$ in $Q_v$ which lies further from $v$ than $u$ on the underlying undirected path of $Q_v$ (if such a $w$ exists), then $w\in \{g^\diamond_1(u),g^\diamond_2(u)\}$.\label{pointthree}
\end{enumerate}
Note that the path $Q_v$ consisting solely of the vertex $v$ satisfies these conditions, so such a path $Q_v$ does exist.

We now pick a subcollection of these paths which are disjoint.  To do this, iteratively, for each $i=r,\ldots,1$, let
\[
\bar{B}_i=\{v\in B_{i}\setminus B_{i-1}:v\notin V(Q_{u})\text{ for each }u\in X\cup \bar{B}_{r}\cup \bar{B}_{r-1}\cup \ldots \cup \bar{B}_{i+1}\}.
\]
Let $\bar{B}=\bar{B}_r\cup\ldots\cup \bar{B}_1$. We will show that the paths $Q_v$, $v\in X\cup \bar{B}$, satisfy the conditions in the lemma. That is, that they are vertex-disjoint and satisfy \emph{\ref{new1}}--\emph{\ref{new2}}.

Note first that, by the choice of $\bar{B}$, the paths $Q_v$, $v\in X\cup \bar{B}$, contain every vertex in $X\cup B$, and thus \emph{\ref{new2}} holds. As \ref{pointone} holds, to show that \emph{\ref{new1}} holds it is sufficient to show that, for each $v\in X\cup B$, the endvertices of $Q_v$ are in $A^+\cup A^-$. Therefore, to complete the proof of the lemma we need only show the following two claims.

\begin{claim}\label{sun1}
For each $v\in X\cup B$, the endvertices of $Q_v$ are in $A^+\cup A^-$.
  \end{claim}
  \begin{claim}\label{sun2}
The paths $Q_v$, $v\in X\cup \bar{B}$, are vertex-disjoint.
  \end{claim}

Before proving these claims, we will deduce two properties, \ref{keykey} and \ref{keykey2} that we require.
For each vertex $v\in X\cup B$, let $j_v$ be the largest integer $i\in \{0,1,\ldots,r\}$ such that $v\in X\cup (B\setminus B_i)$. Note that, if $v\in X\cup B$, $\diamond\in \{+,-\}$, and $x,y\in V(Q_v)$, and $x$ is closer to $v$ on $Q_v$ than $y$ (and possibly even $x=v$) and $y$ is a $\diamond$-neighbour of $x$ on $Q_v$, then, by  \ref{pointthree}, $y\in \{g_1^\diamond(x),g_2^\diamond(x)\}$, so that $xy\in E(H^\diamond)$. Thus, by the definition of $H^\diamond$, if, in addition, $x,y\in X\cup B$, then $j_y<j_x$. Thus, we have the following two properties.

\stepcounter{propcounter}
\begin{enumerate}[label = {\bfseries{\Alph{propcounter}\arabic{enumi}}}]
\item If $v\in X\cup B$, and $x,y\in (X\cup B)\cap V(Q_v)$, are such that $x$ is closer to $v$ on $Q_v$ than $y$, then $j_y<j_x$.\label{keykey}
\item If $v\in X\cup B$, and $x\in (X\cup B)\cap V(Q_v)$, then any vertex on $Q_v$ in $\{g_1^+(x),g_2^+(x),g_1^-(x),g_2^-(x)\}$ is a neighbour of $x$ on $Q_v$ which is further than $v$ from $x$ on $Q_v$, and every such neighbour must be in $\{g_1^+(x),g_2^+(x),g_1^-(x),g_2^-(x)\}$.\label{keykey2}
\end{enumerate}

We now prove Claim~\ref{sun1} and~\ref{sun2}.

  \begin{proof}[Proof of Claim~\ref{sun1}]
 Suppose, to the contrary, that there is some $v\in X\cup B$ such that $Q_v$ has some endvertex, $w$ say, which is not in $A^+\cup A^-$.   Let $P'_{f(v)}\subset P_{f(v)}$ be the subpath of $P_{f(v)}$ of which $Q_v$ is a copy, and let its endvertex of which $w$ is a copy be $w'$. Note that, by \ref{keykey}, there are at most $r-1$ vertices between $v$ and $w$ on $Q_v$. Thus, we can pick $x'\in V(P_{f(v)})\setminus V(P'_{f(v)})$ and $\diamond\in \{+,-\}$ be such that $x'$ is a $\diamond$-neighbour of $w$ on $P_{f(v)}$.
 By \ref{keykey2}, $Q_v$ contains at most one vertex in $\{g_1^+(w),g_2^+(w),g_1^-(w),g_2^-(w)\}\subset B\cup A^+\cup A^-$ (and no such vertex if $w\neq v$). Therefore, we can pick $x\in \{g_1^\diamond(w),g_2^\diamond(w)\}\setminus V(Q_v)$. Noting that $Q'_v=Q_v+wx$ satisfies \ref{pointone}--\ref{pointthree} in place of $Q_v$ contradicts the maximality of $Q_v$.
\end{proof}


\begin{proof}[Proof of Claim~\ref{sun2}]
Suppose, to the contrary, that there are distinct $u,v\in X\cup \bar{B}$, and a vertex $x\in V(Q_u)\cap V(Q_v)$. Assume further that $x$ is as close to $u$ as possible on $Q_u$ subject to $x\in V(Q_u)\cap V(Q_v)$.

Suppose first that $u\neq x$ and $v\neq x$. By \ref{keykey2}, there is some $x_u\in V(Q_u)$ which is closer to $x$ on $Q_u$ than $x$ is and for which $x\in \{g_1^+(x_u),g_2^+(x_u),g_1^-(x_u),g_2^-(x_u)\}$. Similarly, there is some $x_v\in V(Q_v)$ which is closer to $v$ on $Q_v$ than $x$ is and for which $x\in \{g_1^+(x_v),g_2^+(x_v),g_1^-(x_v),g_2^-(x_v)\}$. As the sets $\{g_1^+(w),g_2^+(w),g_1^-(w),g_2^-(w)\}$, $w\in X\cup B$, are disjoint, we have $x_v=x_u$, contradicting the assumption on $x$.

Therefore, by swapping the labels of $u$ and $v$ if necessary, we can assume, as $u$ and $v$ are distinct that $u=x$ and $v\neq x$. By the choice of $\bar{B}_{j_u}$, as $u,v\in \bar{B}$ and $u\in V(Q_v)$, we must have that $j_v\leq j_u$. However, by \ref{keykey}, as $v$ is closer to $v$ on $Q_v$ than $u\in V(Q_v)\setminus \{v\}$ is, we have $j_u<j_v$, a contradiction. This completes the proof of the claim, and hence the lemma.
\renewcommand{\qedsymbol}{$\boxdot$}
\qedhere
\renewcommand{\qedsymbol}{$\square$}
\qedsymbol
\end{proof}
\renewcommand{\qedsymbol}{}
\end{proof}
\renewcommand{\qedsymbol}{$\square$}
\fi


\subsection{P\'osa rotation and extension}\label{subsec:posa}
We will now prove Lemma~\ref{lem-posa}, using a standard implementation of P\'osa's rotation-extension technique with edge sprinkling to find a Hamilton cycle (see, for example,~\cite{bollorandomgraphs}). We include the proof to record the very high probability of success that we need, as well as to show that a Hamilton cycle can be found including any fixed edge $e$, to then have a Hamilton path between the vertices of $e$. We begin by defining an $e$-booster, and a rotation.

\begin{defn}\label{boostdef} For a graph $G$ and edges $e,f\in V(G)^{(2)}=\{uv:u,v\in V(G)\text{ and }u\neq v\}$, $f$ is an $e$-booster for $G$ if either $G+e+f$ has a longer path containing $e$ than $G+e$ does, or $G+e+f$ contains a Hamilton cycle through $e$.
\end{defn}

\begin{defn} Let a path $Q$ in a graph $G$ be $Q=u_0u_1\ldots u_\ell$. Let $e\in E(Q)$ and $0\leq i\leq \ell-1$ with $u_iu_{i+1}\neq e$. Then, we \emph{rotate $Q$ in $G$ with $u_0$ and $e$ fixed using $u_\ell u_i$} to get the path $(Q-u_iu_{i+1})+u_{\ell}u_i$.

Note that this is a $u_0,u_{i+1}$-path containing $e$ with the same vertex set as $Q$.
\end{defn}

We now show that a 10-expander has many boosters.

\begin{lemma}\label{rotatelemma} Let $n\geq 3$. If an $n$-vertex graph $G$ is a 10-expander, and $e\in V(G)^{(2)}$, then $G$ has at least $n^2/10^4$ $e$-boosters.
\end{lemma}
\ifproofdone
\noindent\textbf{Proof of Lemma~\ref{rotatelemma} is complete.}
\else
\begin{proof} Note that, by Definition~\ref{boostdef}, we can assume that $e\in E(G)$. Let $P$ be a maximal path in $G$ containing $e$, and $V=V(P)$. Let $E$ be the set of pairs $ab$ for which there is an $a,b$-path, $P_{ab}$ say, in $G$ containing $e$ with vertex set $V$. Note that each  $ab\in E$ is an $e$-booster for  $G$. Indeed, $P_{ab}+ab$ is a cycle in $G$ with vertex set $V$ which contains $e$. If $V=V(G)$, then $P_{ab}+ab$ is a Hamilton cycle containing $e$. If $V\neq V(G)$, then, as $G$  is a 10-expander, and hence connected, there exists some $x\in V(G)\setminus V$ and $y\in V$ with $xy\in E(G)$. Let $e'$ be an edge of $P_{ab}+ab$ containing $y$ which is not $e$. Then, $P_{ab}+ab+xy-e'$ is a path containing $e$ with length greater than $P$ in $G+ab$. Thus, in both cases, $ab$ is an $e$-booster for $G$.

Suppose then, for contradiction, that $|E|\leq n^2/10^4$. As the pair of endvertices of $P$ is in $E$, $E\neq \emptyset$, so there is some vertex $u$ in a pair in $E$. If $u$ is in more than $n/50$ such pairs, then there are at least $n/50$ vertices in some pair in $E$. Therefore, by averaging if necessary, there is some $u_0\in V$ such that $1\leq |\{v\in V:u_0v\in E\}|\leq n/50$. Let $Q$ be a path in $G$ with vertex set $V$ which contains $e$ and has $u_0$ as an endvertex. Let $V_0\subset V$ be the set of vertices $v\in V\setminus \{u_0\}$ such that a $u_0,v$-path with vertex set $V$ can be reached by iteratively rotating $Q$ in $G$ with $u_0$ and $e$ fixed. Note that $|V_0|\leq |\{v:u_0v\in E\}|\leq n/50$.

For each $v\in V_0$, let $Q_{v}$ be a $u_0,v$-path with vertex set $V$ which can be reached by iteratively rotating $Q$ with $u_0$ and $e$ fixed. For each $v\in V_0$, as $Q_v$ is a maximal path in $G$, we have $N(v)\subset V(Q_v)=V$, and thus $N(V_0)\subset V$.

Let $\ell$ be the length of $Q$ and label $Q$ as $u_0u_1\cdots u_\ell$. Note that, for each $i\in [\ell-1]$, if $u_{i-1},u_i,u_{i+1}\notin V_0$ then in any sequence of rotations fixing $e$ and $u_0$ we always preserve the subpath $u_{i-1}u_iu_{i+1}$ (in either order), and therefore we never rotate using an edge containing $u_i$. Thus, for each $i\in [\ell-1]$, if $v\in V_0$ and $u_iv\in E(G)$, then we must have one of $u_{i-1}\in V_0$, $u_i\in V_0$, $u_{i+1}\in V_0$, $u_{i-1}u_i=e$ or $u_iu_{i+1}= e$. There can be at most $3|V_0|+2$ values of $i$ for each $v$, and hence $|N_G(V_0)|\leq 3|V_0|+2< 10|V_0|$. As $G$ is a 10-expander, $|V_0|> n/20$, contradicting that $|V_0|\leq n/50$.
\end{proof}
\fi

We use the following standard form of Azuma's inequality for a sub-martingale (see, for example~\cite{probmethod}, for an exposition of martingales and Azuma's inequality).

\begin{theorem}[Azuma's inequality]\label{azuma} If $X_0,X_1,\ldots,X_n$ is a sub-martingale, and $|X_i-X_{i-1}|\leq c_i$ for each $1\leq i\leq n$, then, for each $t>0$, $\P(X_n-X_0\leq -t)\leq \exp\left(\frac{-t^2}{2\sum_{i=1}^nc_i^2}\right)$.
\end{theorem}

We now prove Lemma~\ref{lem-posa}. 

\ifproofdone
\noindent\textbf{Proof of Lemma~\ref{lem-posa} is complete.}
\else
\begin{proof}[Proof of Lemma~\ref{lem-posa}] As in the lemma statement, let $G_0$ be a 10-expander with vertex set $[n]$ and let $x,y\in V(G_0)$ be distinct. Let $p=\log n/(10^5n)$ and $G_1= G(n,p)$. We will show that, with probability $1-\exp(-\omega(n))$, $G_0\cup G_1$ contains a Hamilton $x,y$-path, and thus the lemma follows.

Let $m=E(G_1)$ and, uniformly at random, label $E(G_1)=\{e_1,\ldots,e_{m}\}$. Let $m_0=\log n/10^6$. By a simple application of Lemma~\ref{chernoff}, we have
\begin{equation}\label{msize}
\P(m\geq m_0)=1-\exp(-\omega(n)).
\end{equation}

We will show that, letting $E$ be the event that $G_0\cup G_1$ contains an $x,y$-Hamilton cycle, we have
\begin{equation}\label{Emm}
\P(E|m=m_0)=1-\exp(-\omega(n)).
\end{equation}
As $\P(E|m=\bar{m})\geq \P(E|m=m_0)$ for each $\bar{m}\geq m_0$, we will then have
\begin{align*}
\P(E)&=\sum_{\bar{m}=0}^{\binom{n}{2}}\P(m=\bar{m})\cdot \P(E|m=\bar{m})\geq \sum_{\bar{m}=m_0}^{\binom{n}{2}}\P(m=\bar{m})\cdot \P(E|m=m_0)
\\
&
\geq \P(m\geq m_0)\cdot \P(E|m=m_0)\overset{\eqref{msize},\eqref{Emm}}{=}1-\exp(-\omega(n)),
\end{align*}
as required.

It is left then to prove \eqref{Emm}. Let $H_0=G_0$, and, for each $1\leq i\leq m_0$, let $H_i=H_{i-1}+e_i$ and let $X_i$ be 1 if $e_i$ is an $xy$-booster for $H_{i-1}$, and 0 otherwise. For each $i\in [m_0]$, as $H_{i-1}$ contains $H_0$ it is always a 10-expander. Hence, by Lemma~\ref{rotatelemma}, the set of $xy$-boosters for $H_{i-1}$ always has size at least $n^2/10^4$. Therefore, the probability that $X_i=1$, conditioned on any possible values of $H_1,\ldots,H_{i-1}$ is at least $(n^2/10^4-i)/\binom{n}{2}\geq 1/10^5$.
 For each $0\leq i\leq m_0$, let $Y_i=\sum_{j=1}^i(X_j-1/10^5)$, and note that $Y_0,Y_1,\ldots,Y_{m_0}$ is a sub-martingale with $|Y_i-Y_{i-1}|\leq 1$ for each $i\in [m_0]$

Thus, by Theorem~\ref{azuma},
\[
\P(Y_{m_0}<-m_0/10^6)= \exp(-\Omega(m_0))=\exp(-\omega(n)).
\]
Note that, if $Y_{m_0}\geq -m_0/10^6$, then $\sum_{i=0}^{m_0} X_i\geq 9{m_0}/10^6=\omega(n)$.
Furthermore, if $\sum_{i=0}^{m_0} X_i\geq n$, then at least $n$  $xy$-boosters are added somewhere in the sequence $H_1,\ldots,H_{m_0}$, and hence $H_{m_0}$ contains a Hamilton cycle containing $xy$, and thus a Hamilton $x,y$-path. Therefore, \eqref{Emm} holds, and the proof is complete.
\end{proof}
\fi



\section{Proof of Theorem~\ref{thmrandprocess}}\label{sec:di}
We will prove Theorem~\ref{thmrandprocess} from Theorem~\ref{mainthm} in Section~\ref{sec:randpf}, and then deduce Theorems~\ref{hittime} and~\ref{sharpthres} in Sections~\ref{sec:imp1} and~\ref{sec:imp2}. We start in Section~\ref{sec:diprop} by proving some properties of the random digraph process that we require, before selecting the vertices from the cycle to be embedded in Section~\ref{subsec:choosevx}.

\subsection{Properties of the random digraph process}\label{sec:diprop}
In the $n$-vertex random digraph process $D_0,D_1,\ldots,D_{n(n-1)}$, we focus on four particular digraphs, $D_{i_0}$, $D_{i_1}$, $D_{i_2}$ and $D_{i_3}$, where
\begin{equation}\label{iseq}
i_0=\frac{9n\log n}{20}, \;\;\;  i_1=\frac{n\log n}{2}-n\log\log n,\;\;\;  i_2=\frac{3n\log n}{4}\;\;\;  \text{and} \;\;\; i_3=n\log n+2n\log\log n.
\end{equation}
We will prove that the properties \emph{\ref{rp1}}--\emph{\ref{rp7}} typically hold for these digraphs, in Lemmas~\ref{lem:keyrandprops} and~\ref{lem:keyrandprops2}, where, for convenience, we use the random digraphs $K_0,K_1,K_2$ and $K_3$.
The digraphs we need to consider for Theorem~\ref{thmrandprocess} will lie between $D_{i_1}$ and $D_{i_3}$. That is, $D_{i_1}$ has, with high probability, some vertex $v$ with $d^+(v)+d^-(v)\leq 1$ (see \emph{\ref{rp7other}}), and $D_{i_3}$ has, with high probability, minimum in- and out-degree at least 2 (see \emph{\ref{rp7}}). We use $D_{i_0}$ as a point in the random digraph process by which most vertices do not have low in- or out-degree, and $D_{i_2}$ as a point by which the low degree vertices are likely to be few enough that they are well-spaced in the digraph.
 In what follows, we consider `low degree' to be degree at most $\log n/300$.

After adding $i_0$ edges in the random digraph process, giving the digraph $D_{i_0}$, it will typically already be clear that most of the vertices in each $D_i$, $i_1\leq i\leq i_3$, will not have low in- or out-degree. We collect the vertices for which this is not yet guaranteed into the set $S_0$ (see \eqref{sjdefn}), which is likely to have size at most $n^{2/3}$ (see \emph{\ref{rp1}}). To determine the vertices of low in- and out-degree of each  $D_i$, $i_1\leq i\leq i_3$, we only need to consider the edges in $E(D_i)\setminus E(D_{i_0})$ with at least one vertex in $S_0$ -- say the set of these edges is $E_i$. Conditioning on $E(D_{i_0})$ and $E_i$, let $m_i=|E(D_{i_0})\cup E_i|$. Observe that, with this conditioning, $E(D_i)\setminus (E(D_{i_0})\cup E_i)$ is distributed as $i-m_i$ edges chosen uniformly at random from the non-edges in $E(D_{i_0})$ with no vertex in $S_0$. As $S_0$ is a sublinear set of vertices, when $i\geq i_1$, we are likely to have that $i-m_i=\Omega(\log n)$ (see \emph{\ref{rp6}}). These $i-m_i$ edges will provide the random digraph which we use (in a modified form) to apply Theorem~\ref{mainthm}. The pseudorandom digraph in this application will (suitably modified) be the digraph $D_i$ with the addition of any edge from $E(D_i)\setminus E(D_{i_0})$ with at least one vertex in $S_0$. As a technique, conditioning on (events including) $S_0$ in this way comes from work of Krivelevich, Lubetzky and Sudakov~\cite{krivcores} in their study of the Hamiltonicity of the $k$-core.

We use $i_2$ as an arbitrary midpoint in the interval $[i_1,i_3]$ by which the relevant structure of the random digraph will have changed. In $D_{i_1}$, we expect to have plenty of vertices with out-degree 0 or with in-degree 0, and some vertices with out- and in-degree both 1. Our methods are complicated by the likely existence of edges (and short paths) between the vertices with low in- or out-degree (we give a likely upper-bound for such paths in \emph{\ref{rp2}}). We deal with this by showing that, for $i_1\leq i\leq i_2$, $D_i$ will likely have sufficiently many vertices with in-degree 0 (see \emph{\ref{rp4}}) that any cycle we embed into $D_i$ has enough changes of direction to cover not only the vertices of in-degree 0 or out-degree 0, but also any vertices in edges and short paths between vertices with low in- or out-degree. Helpfully, there are typically never any edges or short paths between vertices with both low in- and out-degree (see \emph{\ref{rp5a}}), or any short cycles containing a vertex of low in- or low out-degree (see \emph{\ref{rp2other}}).

When enough edges are added to reach $D_{i_2}$, the set of low in- or out-degree vertices will have decreased in size enough that there are likely to be no edges or very short paths between these vertices, even after more edges are added to reach $D_{i_3}$ (see \emph{\ref{rp5b}}). This means that low in- and out-degree vertices are sufficiently far apart in $D_i$, $i_2\leq i\leq i_3$, that we can assign them neighbours without worrying about conflicts.

In addition to the properties mentioned above, we prove properties \emph{\ref{rp3}}--\emph{\ref{rpnew}} are likely to hold,  which we will use to help show pseudorandom properties of a modified subgraph of $D_i$, for each $i_1\leq i\leq i_3$. We start with the following useful proposition.

\begin{prop}\label{prop:simple} There is some $n_0$ such that the following holds for each $n\geq n_0$. Letting $d=\log n/300$, for each $k\leq n^{1/2}$ and $p$ with $\log n/4n\leq p\leq 2\log n/n$, we have
\[
\sum_{i=0}^d\binom{n}{i}p^i(1-p)^{n-k}\leq \exp\left(-pn+\frac{1}{30}\log n\right).
\]
\end{prop}
\begin{proof} As $100<enp/d\leq 600e$,  we have
\begin{align*}
\sum_{i=0}^d\binom{n}{i}p^i(1-p)^{n-k}&\leq \sum_{i=0}^d\left(\frac{enp}{i}\right)^i(1-p)^{n-k}\leq (d+1)\left(\frac{enp}{d}\right)^de^{-p(n-k)}
\\
&\leq \left({700e}\right)^de^{-p(n-k)}\leq \exp(-pn+pk+d\log(700e))\leq \exp(-pn+(\log n)/30),
\end{align*}
where the last line of inequalities hold for sufficiently large $n$ as $pk\leq 2n^{-1/2}\log n$ and $\log(700e)<8$.
\end{proof}

We now prove that properties \emph{\ref{rp1}}--\emph{\ref{rpnew}} are likely to hold, as follows.

\begin{lemma}\label{lem:keyrandprops}
  Let $(K_0,K_1,K_2,K_3)$ be drawn uniformly at random from the set of such tuples such that, for each $j\in \{0,1,2,3\}$, $K_j$ is a digraph with vertex set $[n]$ and $i_j$ edges (as set in~\eqref{iseq}), and $K_0\subset K_1\subset K_2\subset K_3$.
  Let $d=\log n/300$, and, for each $j\in \{0,1,2,3\}$, let
  \begin{equation}\label{sjdefn}
  S_j=\{v\in  [n]:d^+_{K_j}(v)\leq d\;\text{ or }\;d^-_{K_j}(v)\leq d\}.
  \end{equation}
Let
  \[
  T=\{v\in [n]:d^+_{K_1}(v)\leq d\;\text{ and }\;d^-_{K_1}(v)\leq d\}.
  \]

Then, with high probability, the following hold.
\stepcounter{propcounter}
\begin{enumerate}[label = \textbf{\emph{\Alph{propcounter}\arabic{enumi}}}]
\item $|S_0|\leq n^{2/3}$.\label{rp1}
\item $E(K_3)\setminus E(K_0)$ contains at most $n$ edges with some vertex in $S_0$.\label{rp6}
\item The number of paths in $K_3$ between vertices in $S_1$ with length at most 4 is at most $n^{1/6}$.\label{rp2}
\item There are no cycles in $K_3$ with length at most 3 containing a vertex in $S_1$.\label{rp2other}
\item $K_3[S_2\cup N_{K_3}^+(S_2)\cup N_{K_3}^-(S_2)]$ is the disjoint union of $|S_2|$ stars.\label{rp5b}
\item $K_3[T\cup N_{K_3}^+(T)\cup N_{K_3}^-(T)]$ is the disjoint union of $|T|$ stars with no vertices in $S_1\setminus T$.\label{rp5a}
\item Each $v\in [n]$ has at most $2$ in- or out-neighbours in $K_3$ in $S_1\cup N^+_{K_3}(S_1-v)\cup N^-_{K_3}(S_1-v)$.\label{rp3}
\item For any sets $A,B\subset [n]$ and $\diamond\in \{+,-\}$ with $|A|\leq 100n\log\log n/\log n$ and, for each $v\in A$, $d^\diamond_{K_3}(v,B)\geq (\log n)^{2/3}/2$, we have $|B|\geq 100|A|(\log n)^{1/3}$.\label{rp8}
\item $\Delta^\pm(K_3)\leq 50\log n$.\label{rpnew}
\end{enumerate}
\end{lemma}

\ifproofdone

\bigskip

\noindent\textbf{Proof of Lemma~\ref{lem:keyrandprops} is complete.}
\else

\begin{proof} First, we will choose binomial random digraphs $\bar{K}_0,\bar{K}_1, \bar{K}_2$, and $\bar{K}_3$ and use them to choose the digraphs $K_0,  K_1, K_2,$ and $K_3$ with the distribution in the lemma. This will allow us to prove likely properties of  $\bar{K}_0,\bar{K}_1, \bar{K}_2$, and $\bar{K}_3$ and
infer \emph{\ref{rp1}}--\emph{\ref{rpnew}} from them.

Let $N=n(n-1)$, $p_0=(i_0-n)/N$, $p_1=(i_1-n)/N$, $p_2=(i_2-n)/N$, and $p_3=(i_3+n)/N$.
For each $u,v\in [n]$ with $u\neq v$, let $X_{uv}$ be chosen uniformly at random from $[0,1]$. For each $j\in \{0,1,2,3\}$, let $\bar{K}_j$ be the digraph with vertex set $[n]$ and edge set $\{uv:u,v\in [n],u\neq v,X_{uv}\leq p_j\}$.
Let $d=\log n/300$. For each $j\in \{0,1,2,3\}$, let
\[
\bar{S}_j=\{v\in [n]:d^+_{\bar{K}_j}(v)\leq d\text{ or }d^-_{\bar{K}_j}(v)\leq d\}\]
and
\[
\bar{T}=\{v\in [n]:d^+_{\bar{K}_1}(v)\leq d\text{ and }d^-_{\bar{K}_1}(v)\leq d\}.
\]

Now, note that, for each $j\in \{0,1,2,3\}$, $\bar{K}_j$ has the same distribution as $D(n,p_j)$. Furthermore, $\bar{K}_0\subset\bar{K}_1\subset\bar{K}_2\subset\bar{K}_3$.
Let $\ell=e(\bar{K}_3)-e(\bar{K}_0)$, and label the edges of $E(\bar{K}_3)\setminus E(\bar{K}_0)$ as $e_1,\ldots,e_\ell$ uniformly at random subject to the restriction that the edges of $E(\bar{K}_1)\setminus E(\bar{K}_0)$ come first in this order, followed by those in $E(\bar{K}_2)\setminus E(\bar{K}_1)$, and then those in $E(\bar{K}_3)\setminus E(\bar{K}_2)$.
Let $\mathbf{E}$ be the event that $e(\bar{K}_3)\geq i_3$, and, for each $j\in \{0,1,2\}$, $e(\bar{K}_j)\leq i_j$.
If $\mathbf{E}$ holds, then, for each $j\in \{0,1,2,3\}$, let $K_j$ be the random graph $\bar{K}_0$ with the edges $e_{1},\ldots,e_{i_j-e(\bar{K}_0)}$ added.
If $\mathbf{E}$ does not hold, then let $(K_0,K_1,K_2,K_3)$ be drawn uniformly at random from the set of such tuples such that, for each $j\in \{0,1,2,3\}$, $K_j$ is a digraph with vertex set $[n]$ and $i_j$ edges, and $K_0\subset K_1\subset K_2\subset K_3$. Note that the distribution of $(K_0,K_1,K_2,K_3)$ is the same conditioned on $\mathbf{E}$ holding or on $\mathbf{E}$ not holding, and thus has the distribution as described in the lemma.

Let $\overline{\emph{\ref{rp1}}}$--$\overline{\emph{\ref{rpnew}}}$ be the properties \emph{\ref{rp1}}--\emph{\ref{rpnew}} with $K_j$, $S_j$ and $T$ replaced by $\bar{K}_j$, $\bar{S}_j$ and $\bar{T}$ respectively, for any relevant $j\in \{0,1,2,3\}$.
Now, if $\mathbf{E}$ holds, then $\bar{K}_0\subset K_0$, $\bar{K}_1\subset K_1$, $\bar{K}_2\subset K_2$, and $K_3\subset \bar{K}_3$ and we also have $T\subset \bar{T}$,  and $S_j\subset \bar{S}_j$ for each $j\in \{0,1,2\}$. Therefore, if $\mathbf{E}$ holds, then each property $\overline{\emph{\ref{rp1}}}$--${\overline{\emph{\ref{rpnew}}}}$ implies the corresponding property \emph{\ref{rp1}}--\emph{\ref{rpnew}}. Indeed,
decreasing the sets $T,S_0,S_2$, adding edges in $K_3$ to $K_0$ and removing edges from $K_3$ makes it easier for each of these properties to hold.
Therefore, if $\mathbf{E}$ and $\overline{\emph{\ref{rp1}}}$--${\overline{\emph{\ref{rpnew}}}}$ hold individually with high probability, then \emph{\ref{rp1}}--\emph{\ref{rpnew}} hold collectively with high probability. We now show in turn that each of $\mathbf{E}$ and $\overline{\emph{\ref{rp1}}}$--${\overline{\emph{\ref{rpnew}}}}$ hold with high probability.

\medskip

\noindent\textbf{$\mathbf{E}$:} By Lemma~\ref{chernoff} with $\eps=1/4\log n$, for each $j\in \{0,1,2\}$,
\begin{align}
\P(e(\bar{K}_j)> i_j)&\leq \P(|e(\bar{K}_j)-p_jn(n-1)|> n)\leq \P(|e(\bar{K}_j)-p_jn(n-1)|\geq \eps p_jn(n-1))\nonumber\\
&\leq 2\exp(-\eps^2p_jn(n-1)/3)=o(1).\label{sizecalc}
\end{align}
Similarly, $\P(e(\bar{K}_3)< i_3)=o(1)$. Therefore, $\mathbf{E}$ holds with high probability.

\medskip

\noindent\textbf{$\overline{\emph{\ref{rp1}}}$:}
For each $v\in \bar{S}_0$, there will be some $\diamond\in \{+,-\}$ and $A\subset [n]\setminus \{v\}$ with $|A|\leq d$ such that there is a $\diamond$-edge from $v$ to each vertex in $A$ and no $\diamond$-edge from $A$ to $[n]\setminus (A\cup\{v\})$. Therefore, for large $n$, using Proposition~\ref{prop:simple}, we have, as $p_0n=(9/20-o(1))\log n$, that
\begin{align*}
\E|\bar{S}_0|&\leq n\cdot 2\cdot \sum_{i=0}^d\binom{n-1}{i}p_0^i(1-p_0)^{n-i-1}\leq 2n \cdot \exp(-p_0n+(\log n)/30)=o(n^{2/3}).
\end{align*}
Therefore, by Markov's inequality, with high probability \textbf{$\overline{\emph{\ref{rp1}}}$} holds.

\medskip

\noindent\textbf{$\overline{\emph{\ref{rp6}}}$:} We will show that $\P(\overline{\emph{\ref{rp6}}}\text{ holds}\,|\,\overline{\emph{\ref{rp1}}}\text{ holds})=1-o(1)$. Revealing the edges of $\bar{K}_0$,
if $\overline{\emph{\ref{rp1}}}$ holds, then there are at most $2|\bar{S}_0|\cdot n\leq 2n^{5/3}$ directed non-edges in $\bar{K}_0$ with at least one vertex in $\bar{S}_0$. For each such $uv$, the probability that $uv\in E(\bar{K}_3)\setminus E(\bar{K}_0)$ is $\P(X_{uv}\leq p_3|X_{uv}>p_0)\leq 2p_3$.
Thus, the number of edges in $E(\bar{K}_3)\setminus E(\bar{K}_0)$ has expectation at most $4p_3n^{5/3}=o(n)$. Therefore,  $\P(\overline{\emph{\ref{rp6}}}\text{ holds}\,|\,\overline{\emph{\ref{rp1}}}\text{ holds})=1-o(1)$.
Thus, as $\overline{\emph{\ref{rp1}}}$ holds with high probability, so does $\overline{\emph{\ref{rp6}}}$.

\medskip

\noindent\textbf{$\overline{\emph{\ref{rp2}}}$:} Let $X_1$ be the number of paths of length at most 4 in $\bar{K}_3$ between vertices in $\bar{S}_1$. For each such path, there are distinct vertices $x,y$ (for the endvertices) an integer $k\in \{0,1,2,3\}$, and distinct vertices $v_1,\ldots,v_k\in [n]\setminus\{x,y\}$, so that $xv_1\ldots v_ky$ is a path in $\bar{K}_3$ (with any orientations on its edges), and (as $x,y\in \bar{S}_1$) sets $A_x,A_y\subset [n]\setminus \{x,y,v_1,\ldots,v_k\}$ with size at most $d$ and $\diamond_x,\diamond_y\in \{+,-\}$ for which, for each $v\in \{x,y\}$  there is an $\diamond_v$-edge from $v$ to each vertex in $A_v$ in $\bar{K}_1$ and no $\diamond_v$-edge
 from $v$ to $[n]\setminus (A_v\cup\{x,y,v_1,\ldots,v_k\})$ in $\bar{K}_1$. Thus, as there are $2^{k+1}$ possible orientations for a path with $k$ interior vertices, for large $n$, using Proposition~\ref{prop:simple}, we have, as $2p_1n=(1-o(1))\log n$, that
\begin{align}
\E X_1&\leq \frac{n(n-1)}{2}\cdot \sum_{k=0}^3\binom{n-2}{k}2^{k+1}p_3^{k+1}\cdot \left(2\sum_{i=0}^{d}\binom{n-2-k}{i}p_1^{i}(1-p_1)^{n-2-k-i}\right)^2\nonumber\\
&\leq n\cdot \sum_{k=0}^3(2p_3n)^{k+1} \cdot 4\exp(-2p_1n+(\log n)/15)\nonumber\\
&\leq 16n\cdot (4\log n)^4\cdot \exp(-(9/10+o(1))\log n)=o(n^{1/6}).\label{eq:key}
\end{align}
Thus, with high probability \textbf{$\overline{\emph{\ref{rp2}}}$} holds.

\medskip

\noindent\textbf{$\overline{\emph{\ref{rp2other}}}$:}
Let $X_2$ be the number of cycles of length at most 3 in $\bar{K}_3$ with a vertex in $\bar{S}_1$.
Thus, as there are $2^{k+1}$ possible orientations for a cycle with $k+1$ vertices, for large $n$, using Proposition~\ref{prop:simple}, we have
\begin{align}
\E X_2&\leq n\cdot \sum_{k=1}^2\binom{n-1}{k}2^{k+1}p_3^{k+1}\cdot \left(2\sum_{i=0}^{d}\binom{n-1-k}{i}p_1^{i}(1-p_1)^{n-1-k-i}\right)\nonumber\\
&\leq \sum_{k=1}^2(2p_3n)^{k+1} \cdot 2\exp(-p_1n+(\log n)/30)\leq 4\cdot (4\log n)^3\cdot \exp(-(\log n)/4)=o(1),\label{eq:key2}
\end{align}
where we have used that $p_1n=(1/2-o(1))\log n$.
Therefore,  with high probability \textbf{$\overline{\emph{\ref{rp2other}}}$} holds.

\medskip

\noindent\textbf{$\overline{\emph{\ref{rp5b}}}$:}  Let $X_3$ be the number of paths of length at most 4 in $\bar{K}_3$ between vertices in $\bar{S}_2$. By a similar calculation to \eqref{eq:key}, we have $\E X_3\leq 16n\cdot (4\log n)^4 \cdot \exp(-2p_2n+(\log n)/15)=o(1)$.
Let $X_4$ be the number of cycles of length at most 3 in $\bar{K}_3$ between vertices in $\bar{S}_2$.
By a similar calculation to \eqref{eq:key2}, we have $\E X_4\leq 4\cdot (4\log n)^3\cdot \exp(-p_2n+(\log n)/30)=o(1)$.
Therefore, with high probability $X_3=X_4=0$, and thus $\overline{\emph{\ref{rp5b}}}$ holds.

\medskip

\noindent\textbf{$\overline{\emph{\ref{rp5a}}}$:} First, let $X_5$ be the number of paths with length at most 3 between a vertex in $\bar{T}$ and a vertex in $\bar{S}_1$ in $\bar{K}_3$. Similarly to the analysis for \eqref{eq:key}, we have
\begin{align*}
\E X_5&\leq {n(n-1)}\cdot \sum_{k=0}^2\binom{n-2}{k}2^{k+1}p_3^{k+1}\cdot 2\left(\sum_{i=0}^{d}\binom{n-2-k}{i}p_1^{i}(1-p_1)^{n-2-k-i}\right)^3\\
&\leq 6n\cdot (4\log n)^3\cdot \exp(-3p_1n+(\log n)/10)=o(1).
\end{align*}
Let $X_6$ be the number of cycles with length at most 3 in $\bar{K}_3$ containing some vertex in $\bar{T}$. Then, similarly to the analysis for \eqref{eq:key2}, we have
\begin{align*}
\E X_6&\leq n\cdot \sum_{k=1}^2\binom{n-1}{k}(2p_3)^{k+1} \cdot \left(\sum_{i=0}^{d}\binom{n-1-k}{i}p_1^{i}(1-p_1)^{n-1-k-i}\right)^2\\
&\leq 2\cdot (4\log n)^3\cdot \exp(-2p_1n+(\log n)/15)=o(1).
\end{align*}
Therefore, with high probability, $X_5=X_6=0$, and hence $\overline{\emph{\ref{rp5a}}}$ holds.

\medskip

\noindent\textbf{$\overline{\emph{\ref{rp3}}}$:} Let $X_7$ be the number of vertices $v\in [n]$ with at least 3 out- or in -neighbours in $\bar{K}_3$ in $\bar{S}_1\cup N^+_{\bar{K}_3}(\bar{S}_1-v)\cup N^-_{\bar{K}_3}(\bar{S}_1-v)$. For each such vertex $v$, we can pick three different paths, $P_1$, $P_2$ and $P_3$ in $\bar{K}_3$, with length at most 2 which go from $v$ into $\bar{S}_1$. Letting $H=P_1\cup P_2\cup P_3$, we have that $H$ is a tree (with maybe some doubled edges) if the endvertices of the paths which are not $v$ are distinct. If $|V(H)\cap \bar{S}_1|=2$, then we see that $H$ has at least $|H|$ edges, while if $|V(H)\cap \bar{S}_1|=1$, then $H$ has at least $|H|+1$ edges. Thus, deleting edges if necessary, there is some digraph $H\subset \bar{K}_3$ and some $j\in [3]$ so that $v\in V(H)$, $j\leq |H|\leq 7$, $e(H)\geq |H|+2-j$ edges, and $V(H)$ contains $j$ vertices in $\bar{S}_1$.

Therefore, using Proposition~\ref{prop:simple}, we have, for large $n$, that
\begin{align*}
\E X_7&\leq n\cdot \sum_{j=1}^3\sum_{k=j}^7 \binom{n-1}{k-1}\cdot \binom{k(k-1)}{k+2-j}\cdot p_3^{k+2-j}\cdot  \left(\sum_{i=0}^{d}\binom{n-k}{i}p_1^{i}(1-p_1)^{n-i-k}\right)^{j}
\\
&\leq \sum_{j=1}^3\sum_{k=j}^7 (n p_3)^{k+2-j}\cdot n^{j-2}\cdot 2^{k(k-1)} \cdot  \exp(-j\cdot p_1n+j(\log n)/30)
\\
&\leq \log^9 n\cdot \sum_{j=1}^3n^{j-2}\cdot \exp(-(2j/5)\log n)=o(1).
\end{align*}
Therefore, with high probability, $X_7=0$, and thus $\overline{\emph{\ref{rp3}}}$ holds.

\medskip

\noindent\textbf{$\overline{\emph{\ref{rp8}}}$:} Let $t_0=(\log n)^{1/3}/200$, $t_1=100n\log\log n/\log n$, $d_0=(\log n)^{2/3}/2$ and $d_1=100(\log n)^{1/3}$.
Note that if there are some sets $A,B\subset[n]$ satisfying $|B|< d_1|A|$ and, for some $\diamond\in \{+,-\}$, $d^\diamond_{\bar{K}_3}(v,B)\geq d_0$  for each $v\in A$, we have $|B|\geq d_0$, so that $|A|>d_0/d_1=t_0$. Furthermore, for such a pair $A,B$, we can add vertices to $B$ to ensure that $|B|=d_1|A|$. Now, there is some $\diamond\in \{+,-\}$ and $t$ with $t_0\leq t\leq t_1$ and a pair of sets $A,B\subset [n]$ with $|A|=t$, $|B|=d_1t$, and $d^\diamond_{\bar{K}_3}(v,B)\geq d_0$ for each $v\in A$, with probability at most
\begin{align*}
2\sum_{t=t_0}^{t_1}\binom{n}{t}\binom{n}{d_1t}\binom{d_1t}{d_0}^{t}p_3^{d_0t}
&\leq  2\sum_{t=t_0}^{t_1}\left(\left(\frac{en}{t}\right)\cdot \left(\frac{en}{d_1t}\right)^{d_1}\cdot \left(\frac{ed_1tp_3}{d_0}\right)^{d_0}\right)^{t}\\
&\leq  2\sum_{t=t_0}^{t_1}\left(\left(\frac{en}{t}\right)^{d_1+1}\cdot \left(\frac{t(\log n)^{3/4}}{n}\right)^{d_0}\right)^{t}\\
&\leq  2\sum_{t=t_0}^{t_1}\left(e^{d_1+1}\cdot \left(\frac{t}{n}\right)^{d_0-d_1-1}\cdot (\log n)^{3d_0/4}\right)^{t}\\
&\leq  2\sum_{t=t_0}^{t_1}\left(e^{7d_0/8}\cdot \left(\frac{t}{n}\right)^{7d_0/8}\cdot (\log n)^{3d_0/4}\right)^{t}\\
&=  2\sum_{t=t_0}^{t_1}\left(\frac{et(\log n)^{6/7}}{n}\right)^{7d_0t/8}\leq   2\sum_{t=t_0}^{t_1}\left(\frac{1}{2}\right)^{7d_0t/8}\leq  \frac{4}{2^{7d_0t_0/8}}=o(1).
\end{align*}
Therefore, with high probability, $\overline{\emph{\ref{rp8}}}$ holds.

\medskip

\noindent\textbf{$\overline{\emph{\ref{rpnew}}}$:}
Let $X_8$ be the number of vertices in $\bar{K}_3$ with in-degree larger than $50\log n$ or out-degree larger than $50\log n$. Then,
\[
\E X_8\leq 2n\binom{n-1}{50\log n}p_3^{50\log n}\leq 2n\left(\frac{enp_3}{50\log n}\right)^{50\log n}\leq 2n\left(\frac{1}{10} \right)^{50\log n}=o(1).
\]
Therefore, with high probability, $\overline{\emph{\ref{rpnew}}}$ holds.
\end{proof}
\fi

We now prove the properties \emph{\ref{rp7other}}--\emph{\ref{rp7}} are likely to hold. In the following lemma, the random digraphs do not interact in the properties (nor is $K_0$ used), but for convenience we use the same distribution for the random digraphs as in Lemma~\ref{lem:keyrandprops}.

\begin{lemma}\label{lem:keyrandprops2}
  Let $(K_0,K_1,K_2,K_3)$ be drawn uniformly at random from the set of such tuples for which, for each $j\in \{0,1,2,3\}$, $K_j$ is a digraph with vertex set $[n]$ and $i_j$ edges (as set in~\eqref{iseq}), and $K_0\subset K_1\subset K_2\subset K_3$.

Then, with high probability, the following hold.
\begin{enumerate}[label = \textbf{\emph{\Alph{propcounter}\arabic{enumi}}}]\addtocounter{enumi}{9}
\item There is some $v\in [n]$ with $d^+_{K_1}(v)+d^-_{K_1}(v)=0$. \label{rp7other}
\item The number of vertices in $K_2$ with in-degree 0 is at least $n^{1/5}$.\label{rp4}
\item $\delta^\pm(K_3)\geq 2$.\label{rp7}
\end{enumerate}
\end{lemma}

\ifproofdone

\bigskip

\noindent\textbf{Proof of Lemma~\ref{lem:keyrandprops2} is complete.}
\else

\begin{proof}

\medskip

\noindent\textbf{\emph{\ref{rp7other}}:} Let $p_1=(i_1+n)/n(n-1)$. We will construct a random digraph $D_1$ with the same distribution as $D(n,p_1)$. Similarly as for \eqref{sizecalc}, by Lemma~\ref{chernoff} we will then have that $\P(e(D_1)\geq i_1)=1-o(1)$, and thus, it is sufficient to show that, with high probability, there is some $v\in [n]$ with $d^+_{D_1}(v)+d^-_{D_1}(v)=0$.

Let then $q=p_1/(4-2p_1)$ and $p=p_1(2-p_1)$, and let $G=G(n,p)$. Form $D_1$ on the vertex set $[n]$ by taking each edge $uv\in E(G)$ and, independently at random, adding $uv$ but not $vu$ to $E(D_1)$ with probability $(1/2-q)$, adding $vu$ but not $uv$ to $E(D_1)$ with probability $(1/2-q)$, and adding both $uv$ and $vu$ to $E(D_1)$ with probability $2q$.
Note that, as $p(1/2+q)=p_1$ and $2pq=p_1^2$, $D_1$ has the same distribution as $D(n,p_1)$. As $p=(\log n-\omega(1))/n$, with high probability $\delta(G)=0$ (see, for example, Theorems 3.5 and 2.2(ii) in~\cite{bollorandomgraphs}).
Furthermore, any $v\in [n]$ with $d_G(v)=0$ satisfies $d_{D_1}^+(v)=d_{D_1}^-(v)=0$, and thus  \emph{\ref{rp7other}} holds with high probability.

\medskip

\noindent\textbf{\emph{\ref{rp4}}:} Let $p_2=(i_2+n)/n(n-1)$ and $D_2=D(n,p_2)$. Similarly as for \eqref{sizecalc}, by Lemma~\ref{chernoff}, we have that $\P(e(D_2)\geq i_2)=1-o(1)$. Therefore, it is sufficient to show that, with high probability, the number of vertices in $D_2$ with in-degree 0 is at least $n^{1/5}$.
Note that, for each $v\in [n]$, the probability that $d^-_{D_2}(v)=0$ is $(1-p_2)^{n-1}=\exp(-(1-o(1))p_2n)=\exp(-(3/4-o(1))\log n)$. Furthermore, this is independent for each $v\in [n]$, and the expected number of vertices in $[n]$ with $d^-_{D_2}(v)=0$ is $n\cdot \exp(-(3/4-o(1))\log n)=\omega(n^{1/5})$.
 Therefore, by Lemma~\ref{chernoff}, the probability that the number of vertices in $D_2$ with in-degree 0 is at least $n^{1/5}$ is $1-o(1)$. Thus, with high probability, $\emph{\ref{rp4}}$ holds.

\medskip

\noindent\textbf{\emph{\ref{rp7}}:} Let $p_3=(i_3-n)/n(n-1)$, so that $p_3(n-2)=\log n+\log\log n+\omega(1)$. Let $D_3=D(n,p_3)$. Similarly as for \eqref{sizecalc}, by Lemma~\ref{chernoff}, we have that $\P(e(D_3)\leq i_3)=1-o(1)$. Therefore, it is sufficient to show that, with high probability, $\delta^\pm(D_3)\geq 2$.

Let $X$ be the number of vertices $v\in[n]$ with $d^+_{D_3}(v)\leq 1$ or $d^-_{D_3}(v)\leq 1$. Then,
\[
\E X\leq 2n\cdot \sum_{i=0}^1\binom{n-1}{i}p_3^i(1-p_3)^{n-1-i}\leq 4n\cdot (np_3)\cdot \exp(-p_3(n-2))=o(1),
\]
so that, with high probability $X=0$, and hence $\delta^\pm(D_3)\geq 2$. Thus, with high probability, \emph{\ref{rp7}} holds.
\end{proof}
\fi

\subsection{Choosing vertices in the cycle}\label{subsec:choosevx}
When embedding each cycle for Theorem~\ref{thmrandprocess}, we have to use the vertices in the cycle with out-degree 0, 2 and 1 to cover vertices in the random digraph $D_i$ with out-degree 0, in-degree 0, and both in- and out-degree 1, respectively. To cover the other low in- or out-degree vertices in $D_i$, we have more choice. We will select the vertices in the cycle to use with the following lemma, which, furthermore, picks the vertices to be in some linear length subpath of the cycle.

\begin{lemma}\label{pathincycle} There is some $n_0$ such that the following holds for each $n\geq n_0$ and $\lambda\in \N$.
Suppose $C$ is an oriented  $n$-vertex cycle with $\lambda$ vertices with out-degree 0. Let $\mu_0,\mu_2\in \N$ with $\mu_0+\mu_2= \lceil 2\lambda/\log n\rceil$ and let $\mu_1= \lceil (n-2\lambda)/\log n\rceil$.
 Then, there exists a path $P\subset C$ with length at most $n/100$ and vertex sets $Z_0,Z_1,Z_2\subset V(P)$ such that the following hold.
\begin{itemize}
\item The vertices in $Z_0\cup Z_1\cup Z_2$ are pairwise at least $100\log n/\log\log n$ apart on $P$ from each other.
\item For each $i\in \{0,1,2\}$, $|Z_i|=\mu_i$ and each vertex $v\in Z_i$ has out-degree $i$ in $C$.
\end{itemize}
\end{lemma}
\ifproofdone

\bigskip

\noindent\textbf{Proof of Lemma~\ref{pathincycle} is complete.}
\else

\begin{proof} For each $i\in \{0,1,2\}$, let $X_i=\{v\in V(C):d_C^+(v)=i\}$. Let $k=100\log n/\log\log n$ and $\ell=\lfloor n/100\rfloor$. We will first choose the path $P$ using the following claim.

\begin{claim} There is a path $P$ with length $\ell$ such that the following hold.\label{clm:goodP}
 \stepcounter{propcounter}
\begin{enumerate}[label = \textbf{\Alph{propcounter}\arabic{enumi}}]
\item If $\mu_0>0$, then $|V(P)\cap X_0|\geq 1+(\mu_0+\mu_2-1)\log n/10^3$.\label{mu0}
\item If $\mu_1>0$, then $|V(P)\cap X_1|\geq 1+(\mu_1-1)\log n/10^3$.\label{mu1}
\item If $\mu_2>0$, then $|V(P)\cap X_2|\geq 1+(\mu_0+\mu_2-1)\log n/10^3$.\label{mu2}
\end{enumerate}
\end{claim}
\begin{proof}[Proof of Claim~\ref{clm:goodP}]
If $\lambda=0$ or $n-2\lambda=0$, then let $P$ be any path in $C$ with length $\ell$. Note that in the first case $\mu_0+\mu_2=0$ and $|V(P)\cap X_1|=|P|$, and thus \ref{mu0}--\ref{mu2} hold. In the second case, $\mu_1=0$ and $|V(P)\cap (X_0\cup X_2)|=|P|$, and therefore, as any two vertices with out-degree 0 in $C$ on $P$ must have some vertex of in-degree 0 in $C$ between them, we have $|V(P)\cap X_0|,|V(P)\cap X_2|\geq |P|/2-1$. Thus, \ref{mu0}--\ref{mu2} hold in this case.

If $0<\lambda\leq (\log n)/2$, then $\mu_0+\mu_2= 1$, so that either $\mu_0=0$ or $\mu_2=0$. If the first case occurs, then, using $\lambda>0$, let $P$  be any path in $C$ with length $\ell$ with $|V(P)\cap X_2|\geq 1$, and, otherwise, let $P$ be any path in $C$ with length $\ell$ and $|V(P)\cap X_0|\geq 1$. Note that, in each case, we have
\[
|V(P)\cap X_1|\geq \ell+1-2\lambda\geq \frac{n}{200}\geq 1+\frac{\mu_1\log n}{10^3}.
\]
Thus, we have that \ref{mu0}--\ref{mu2} hold.

If $0<n-2\lambda\leq (\log n)/2$, then $\mu_1= 1$. Using that $n-2\lambda>0$, let $P$  be any path in $C$ with length $\ell$ with $|V(P)\cap X_1|\geq 1$. Note that we have $|V(P)\cap (X_0\cup X_2)|\geq \ell+1-(n-2\lambda)\geq n/200$. As any two vertices with out-degree 0 in $C$ on $P$ must have some vertex of in-degree 0 in $C$ between them, we have that
\[
|V(P)\cap X_0|,|V(P)\cap X_2|\geq \frac{n}{450}\geq \frac{\lambda}{450}\geq 1+\frac{(\mu_0+\mu_2)\log n}{10^3}.
\]
Thus, we have that \ref{mu0}--\ref{mu2} hold.

Assume then that $\lambda>(\log n)/2$ and $n-2\lambda>(\log n)/2$. Pick an arbitrary direction of $C$, and label the vertices of $C$ as $v_1,\ldots,v_n$ in this order. For each $i\in [n]$, let $P_i$ be the path $v_iv_{i+1}\ldots v_{i+\ell}$, with addition modulo $n$ in the indices.
For each $i\in [n]$, let $f(i)=\frac{n-2\lambda}{n}(\ell+1)-|V(P_i)\cap X_1|$. Let $f(n+1)=f(1)$.
Note that
\[
\sum_{i\in [n]}f(i)=(n-2\lambda)(\ell+1)-|V(P_i)\cap X_1|(\ell+1)=0.
\]
As $|f(i)-f(i+1)|\leq 1$ for each $i\in [n]$, we can thus choose $j\in [n]$ with $|f(j)|\leq 1$. Then, as $n-2\lambda>(\log n)/2$, we have
\begin{align*}
|V(P_j)\cap X_1|&\geq \frac{n-2\lambda}{n}(\ell+1)-1\geq \frac{n-2\lambda}{100}-1\geq 1+\frac{n-2\lambda}{200}\geq 1+ \frac{(\mu_1-1)\log n}{10^3}
\end{align*}
and, as $|V(P_j)\cap X_1|\leq \frac{n-2\lambda}{n}(\ell+1)+1= (\ell+1)-\frac{2\lambda}{n}(\ell+1)+1$, and $\lambda>(\log n)/2$, we have
\[
|V(P_j)\cap (X_0\cup X_2)|\geq \frac{2\lambda}{n}(\ell+1)-1\geq \frac{2\lambda}{100}-1\geq 3+\frac{2\lambda}{200}\geq 3+\frac{2(\mu_0+\mu_2-1)\log n}{10^3}.
\]
As any two vertices in $P_j$ with out-degree 0 on $C$ must have some vertex of in-degree 0 in $C$ between them in $P_j$, we have that $|V(P_j)\cap X_0|,|V(P_j)\cap X_2|\geq 1+(\mu_0+\mu_2-1)\log n/10^3$.
Let $P=P_j$. In every case, we have now chosen a path $P$ with length $\ell$ such that \ref{mu0}--\ref{mu2} hold.
\renewcommand{\qedsymbol}{$\boxdot$}
\end{proof}
\renewcommand{\qedsymbol}{$\square$}

Given the path $P$ as in Claim~\ref{clm:goodP}, we now pick the sets $Z_0,Z_1$ and $Z_2$. We do this in two cases, according to whether $\mu_0+\mu_2\leq \mu_1$ or $\mu_0+\mu_2> \mu_1$. In each case, we pick the smaller of $Z_0\cup Z_2$ and $Z_1$ first.

\medskip

\noindent\textbf{Case I.}
Suppose that $\mu_0+\mu_2\leq \mu_1$. Pick vertex sets $Z_0\subset X_0$ and $Z_2\subset X_2$ so that the vertices in $Z_0\cup Z_2$ are pairwise at least $k$ apart on $P$, $|Z_0|\leq \mu_0$, $|Z_2|\leq \mu_2$, and, subject to this, $|Z_0\cup Z_2|$ is maximised.
Suppose, for contradiction, that $|Z_0|<\mu_0$. If $|Z_0|=0$ and $\mu_2=0$, then $Z_2=\emptyset$ and, by \ref{mu0}, we have $|V(P)\cap X_0|\geq 1$, so that there is a vertex $z\in V(P)\cap X_0$ which is a distance at least $k$ apart from every vertex in $Z_1\cup Z_2=\emptyset$ on $P$, a contradiction. Therefore, we can assume that $|Z_0|>0$ or $\mu_2>0$. In each case, we have, as $\mu_0>|Z_0|$, that $\mu_0+\mu_2\geq 2$. Hence, by \ref{mu0}, $|V(P)\cap X_0|\geq (\mu_0+\mu_2)\log n/(2\cdot 10^3)$.
Now, every vertex in $V(P)\cap X_0$ is within distance $k-1$ of some vertex in $Z_0\cup Z_2$ on $P$, so that
\[
|V(P)\cap X_0|\leq (2k-1)|Z_0\cup Z_2|< (2k-1)(\mu_0+\mu_2)\leq \frac{10^3(\mu_0+\mu_2)\log n}{\log\log n}.
\]
For sufficiently large $n$ this contradicts $|V(P)\cap X_0|\geq (\mu_0+\mu_2)\log n/(2\cdot 10^3)$. Similarly, we get a contradiction if $|Z_2|<\mu_2$. Therefore, we have $|Z_0|=\mu_0$ and $|Z_2|=\mu_2$.

Let then $Z_1\subset X_1$ be a maximal set subject to $|Z_1|\leq \mu_1$ and that the vertices in $Z_0\cup Z_1\cup Z_2$ are pairwise at least $k$ apart on $P$. Suppose, for contradiction, that $|Z_1|<\mu_1$. Then, every vertex in $X_1$ is within distance $k-1$ of some vertex in $Z_0\cup Z_1\cup Z_2$ on $P$, so that
\[
|V(P)\cap X_1|\leq (2k-1)|Z_0\cup Z_1\cup Z_2|< (2k-1)(\mu_0+\mu_1+\mu_2)\leq (2k-1)\cdot 2\mu_1\leq \frac{10^3\mu_1\log n}{\log\log n}.
\]
As $\mu_1\geq \mu_0+\mu_2$ (and $\mu_0+\mu_1+\mu_2\geq n/\log n$), we have $\mu_1\geq 2$. Therefore, for sufficiently large $n$, this contradicts \ref{mu1}. Thus, we have $|Z_1|=\mu_1$, and $Z_0,Z_1,Z_2$ and $P$ satisfy the conditions in the lemma.

\medskip

\noindent\textbf{Case II.}
Assume then that $\mu_0+\mu_2> \mu_1$. Let $Z_1\subset X_1$ be a maximal set subject to $|Z_1|\leq \mu_1$ and that the vertices in $Z_1$ are pairwise at least $k$ apart on $P$.
Suppose, for contradiction, that $|Z_1|<\mu_1$. If $|Z_1|=0$, then $\mu_1> 0$ and, by \ref{mu1}, we have $|V(P)\cap X_1|\geq 1$, so that there is a vertex $z\in V(P)\cap X_1$ which is a distance at least $k$ apart from every vertex in $Z_1=\emptyset$, a contradiction. Therefore, we can assume that $\mu_1>|Z_1|>0$, and hence, by \ref{mu1}, that $|V(P)\cap X_1|\geq \mu_1\log n/(2\cdot 10^3)$. Then, every vertex in $X_1$ is within distance $2k-1$ of $Z_1$ on $P$, so that
\[
|V(P)\cap X_1|\leq (2k-1)|Z_1|< (2k-1)\mu_1\leq \frac{10^3\mu_1\log n}{\log\log n}.
\]
For sufficiently large $n$ this contradicts $|V(P)\cap X_1|\geq \mu_1\log n/(2\cdot 10^3)$. Therefore, we have $|Z_1|=\mu_1$.

Pick vertex sets $Z_0\subset X_0$ and $Z_2\subset X_2$ so that the vertices in $Z_0\cup Z_1\cup Z_2$ are pairwise at least $k$ apart on $P$, $|Z_0|\leq \mu_0$, and $|Z_2|\leq \mu_2$, and, subject to this, $|Z_0\cup Z_2|$ is maximised.
Suppose, for contradiction, that $|Z_0|<\mu_0$. Then, as every vertex in $V(P)\cap X_0$ is within distance $k-1$ of $Z_0\cup Z_1\cup Z_2$ on $P$, and $\mu_0+\mu_2> \mu_1$, we have
\begin{equation}\label{eqn:lastone}
|V(P)\cap X_0|\leq (2k-1)|Z_0\cup Z_1\cup Z_2|< (2k-1)(\mu_0+\mu_1+\mu_2)\leq \frac{10^3(\mu_0+\mu_2)\log n}{\log\log n}.
\end{equation}
As $\mu_0+\mu_2> \mu_1$ (and $\mu_0+\mu_1+\mu_2\geq n/\log n$), we have $\mu_0+\mu_2\geq 2$. Thus, by \ref{mu0}, $|V(P)\cap X_0|\geq (\mu_0+\mu_2)\log n/(2\cdot 10^3)$, which, for large $n$, contradicts \eqref{eqn:lastone}. Similarly, we get a contradiction if $|Z_2|<\mu_2$. Therefore, we have $|Z_0|=\mu_0$, $|Z_2|=\mu_2$, and thus $Z_0,Z_1,Z_2$ and $P$ satisfy the conditions in the lemma.
\end{proof}
\fi


\subsection{Proof of Theorem~\ref{thmrandprocess} from Theorem~\ref{mainthm}}\label{sec:randpf}
We can now prove Theorem~\ref{thmrandprocess}, which we restate for convenience.
\thmrandprocess*
 In the initial set-up of the proof, we follow the explanation at the start of Section~\ref{sec:diprop}, and use the values of $i_0,i_1,i_2$ and $i_3$ in \eqref{iseq}. In the random digraph process $D_0,D_1,\ldots,D_{n(n-1)}$, we condition on the value of $D_{i_0}$, and, for each $i$ with $i_0\leq i\leq i_3$, on the edges in $E(D_i)\setminus E(D_{i_0})$ with at least one vertex in $S$, where $S$ is the set of low in- or out-degree vertices in $D_{i_0}$. For each $i$ with $i_1\leq i\leq i_2$, subject to the conditioning in the previous sentence, we know which cycles we want to embed into $D_i$ (gathered into the set $\mathcal{C}_i$) as $s_i$ and $t_i$ are determined by this conditioning. The main part of the proof consists of showing that, subject to the conditioning, a copy of each such cycle $C$ is very likely to appear in $D_i$ (see Claim~\ref{cl-forunion}), using Theorem~\ref{mainthm}.
 Essentially, the edges revealed in the conditioning allow the bound on the probability of the appearance of a copy of $C\in \mathcal{C}_i$ to be much stronger than the probability a copy of $C$ appears in $D_i$ (see also Section~\ref{sec:guide}). Thus, the bound given by Theorem~\ref{mainthm} is strong enough to take a union bound over all the cycles in $\mathcal{C}_i$ (and $i$ with $i_1\leq i\leq i_2$) to complete the proof.

As highlighted by paragraph titles, the proof that a relevant cycle $C$ is very likely to appear in $D_i$ subject to the conditioning (that is, the proof of Claim~\ref{cl-forunion}) proceeds in the following steps.

\begin{itemize}
\item We simplify the random edges in $D_i$ which have not been conditioned on, replacing them with a binomial random digraph with vertex set $[n]\setminus S$ (called $\bar{D}$).
\item We identify the low in- and out-degree vertices in $D_i$ (which form a subset of $S$), and partition them according to the degree of the vertex in the cycle which will be embedded to them, before choosing vertices to embed as their neighbours to get paths of length 2.
\item We choose the subpath $P$ of $C$ so that we can embed well-spaced paths of length 2 from $P$ to these paths of length 2 in $D_i$ (using Lemma~\ref{pathincycle}).
\item We modify the conditioned edges of $D_i$ (those in the graph $H_i$) to get a pseudorandom digraph ($H_i'$), by contracting the chosen paths of length 2 in $D_i$ (and possibly altering some edges).
\item We modify $C$ and $P$ accordingly.
\item We modify the binomial random digraph $\bar{D}$ with vertex set $[n]\setminus S$ accordingly.
\item We apply Theorem~\ref{mainthm} to these modified digraphs, before undoing the modifications to find a copy of the cycle $C$ in $D_i\cup \bar{D}$.
\end{itemize}

\ifproofdone

\bigskip

\noindent\textbf{Proof of Theorem~\ref{thmrandprocess} is complete.}
\else
\begin{proof}[Proof of Theorem~\ref{thmrandprocess}] 
Let $d=\log n/300$ and let $i_0,i_1,i_2$ and $i_3$ be as in \eqref{iseq}. Let $D_0,D_1,\ldots,D_{n(n-1)}$ be the $n$-vertex random digraph process. First, note that, by \emph{\ref{rp7other}} in Lemma~\ref{lem:keyrandprops}, with high probability there is some
$v\in [n]$ with $d_{D_{i_1}}^+(v)=d_{D_{i_1}}^-(v)=0$, and thus, for each $0\leq i\leq i_1$ the property in Theorem~\ref{thmrandprocess} for $D_i$ holds trivially.
Similarly, by \emph{\ref{rp7}} in Lemma~\ref{lem:keyrandprops}, we have, with high probability that $\delta^\pm(D_{i_3})\geq 2$, and hence $s_{i_3}=t_{i_3}=0$, and thus we need to show that, with high probability, $D_{i_3}$ contains a copy of every $n$-vertex oriented cycle. If this holds, then $D_i$ contains a copy of every $n$-vertex oriented cycle for each $i\geq i_3$. Therefore, it is sufficient  show that, with high probability, the property in Theorem~\ref{thmrandprocess} holds for each $i_1\leq i\leq i_3$.

Let
\[
S=\{v\in [n]:d^+_{D_{i_0}}(v)\leq d\;\text{ or }\;d^-_{D_{i_0}}(v)\leq d\}.
\]
For each $i_0\leq i\leq i_3$, let $D'_i$ be the (random) digraph with vertex set $[n]$ and edge set
\[
E(D_{i_0})\cup \{xy\in E(D_{i}):\{x,y\}\cap S\neq\emptyset\},
\]
so that $D'_i\subset D_i$, and $E(D_i)\setminus E(D'_i)\subset ([n]\setminus S)\times ([n]\setminus S)$.

Now, for each $j\in \{0,1,2,3\}$, let ${K}_j=D_{i_j}$, $\bar{K}_j=D'_{i_j}$,
\[
S_j=\{v\in [n]:d^+_{K_{j}}(v)\leq d\;\text{ or }\;d^-_{K_{j}}(v)\leq d\}=\{v\in [n]:d^+_{\bar{K}_{j}}(v)\leq d\;\text{ or }\;d^-_{\bar{K}_{j}}(v)\leq d\},
\]
where we have used that $S_0=S$ and $d^\diamond_{K_j}(v)=d^\diamond_{\bar{K}_j}(v)$ for each $v\in S$, $\diamond\in \{+,-\}$ and $j\in \{0,1,2,3\}$. Furthermore, let
\[
T=\{v\in [n]:d^+_{K_1}(v)\leq d\;\text{ and }\;d^-_{K_1}(v)\leq d\}=\{v\in [n]:d^+_{\bar{K}_1}(v)\leq d\;\text{ and }\;d^-_{\bar{K}_1}(v)\leq d\},
\]
and let $\overline{\emph{\ref{rp1}}}$--$\overline{\emph{\ref{rpnew}}}$ and $\overline{\emph{\ref{rp4}}}$ be the properties \emph{\ref{rp6}}--\emph{\ref{rpnew}} and \emph{\ref{rp4}} with $K_j$ replaced by $\bar{K}_j$ for any relevant $j\in \{0,1,2,3\}$. As $\bar{K}_0= K_0$ and $\bar{K}_3\subset K_3$, each of
\emph{\ref{rp2}}--\emph{\ref{rpnew}} and \emph{\ref{rp4}} implies the corresponding property $\overline{\emph{\ref{rp2}}}$--$\overline{\emph{\ref{rpnew}}}$ and $\overline{\emph{\ref{rp4}}}$. Therefore, by Lemma~\ref{lem:keyrandprops}, we have that, with high probability, $\overline{\emph{\ref{rp1}}}$--$\overline{\emph{\ref{rpnew}}}$ and $\overline{\emph{\ref{rp4}}}$ hold.
Let $\mathcal{D}=(D'_{i_0}, D'_{i_0+1},\ldots,D'_{i_3})$. Let $\mathbf{H}$ be the set of possible values of $\mathcal{D}$ for which $\overline{\emph{\ref{rp1}}}$--$\overline{\emph{\ref{rpnew}}}$ and $\overline{\emph{\ref{rp4}}}$ hold, so that $\P(\mathcal{D}\in \mathbf{H})=1-o(1)$.

Let $\mathcal{C}$ be the set of all oriented cycles whose underlying cycle is the canonical cycle with vertex set $[n]$. For each $i_1\leq i\leq i_3$, let $s_{i}$ be the number of vertices in $D_{i}$ with in-degree or out-degree 0 and let $t_{i}$ be the number of vertices in $D_i$ with in-degree 1 and out-degree 1. Let $\mathcal{C}_{i}$ be the set of cycles in $\mathcal{C}$ with at least $1+(s_{i}-1)\log n$ and at most $n-1-(t_{i}-1)\log n$ changes of direction.
We will show the following claim, where we note that, in \eqref{eq:clmeq}, $d^+_{D_i}(v)+d^-_{D_i}(v)\geq 2$ and $C\in \mathcal{C}_i$ are included because, even though whether they hold or not is decided by the conditioning, they affect the probability for $C\not\subsetsim D_i$. 

\begin{claim}\label{cl-forunion} There is some $n_0$ such that, if $n\geq n_0$, then, for each $\mathcal{H}\in \mathbf{H}$, $i_1\leq i\leq i_3$, and $C\in \mathcal{C}$,
\begin{equation}
\label{eq:clmeq}
\P(d^+_{D_i}(v)+d^-_{D_i}(v)\geq 2\text{ for each }v\in [n]\text{, and }C\in \mathcal{C}_{i}\text{, and }C\not\subsetsim D_i|\mathcal{D}=\mathcal{H})\leq 2e^{-n}.
\end{equation}
\end{claim}

This claim is sufficient to prove the theorem. Indeed, let $E$ be the event that, for some $i_1\leq i\leq i_3$ and $C\in \mathcal{C}$, we have $d^+_{H_i}(v)+d^-_{H_i}(v)\geq 2$ for each $v\in [n]$ and $C\in \mathcal{C}_{i}$ and $C\not\subsetsim D_i$. Then, if $n\geq n_0$, by a simple union bound and Claim~\ref{cl-forunion}, we have, for each $\mathcal{H}\in \mathbf{H}$, that $\P(E|\mathcal{D}=\mathcal{H})\leq (i_3-i_1+1)\cdot 2^n\cdot 2e^{-n}\leq n^2\cdot 2^n\cdot 2e^{-n}$ if $n\geq n_0$, and hence
\begin{align*}
\P(E) &\leq \P(\mathcal{D}\notin \mathbf{H})+\sum_{\mathcal{H}\in\mathbf{H}}\P(\mathcal{D}=\mathcal{H})\cdot \P(E|\mathcal{D}=\mathcal{H})\\
&\leq \P(\mathcal{D}\notin \mathbf{H})+n^2\cdot 2^n\cdot 2e^{-n}\cdot \sum_{\mathcal{H}\in\mathbf{H}}\P(\mathcal{D}=\mathcal{H})\leq \P(\mathcal{D}\notin \mathbf{H})+n^2\cdot 2^n\cdot e^{-n}=\P(\mathcal{D}\notin \mathbf{H})+o(1).
\end{align*}
Thus, as $\P(\mathcal{D}\notin \mathbf{H})=o(1)$, we have that, with high probability, $E$ does not hold, and therefore the property in the theorem holds. It is left then to prove Claim~\ref{cl-forunion}.

\begin{proof}[Proof of Claim~\ref{cl-forunion}] Let $\mathcal{H}=(H_{i_0},\ldots,H_{i_3})\in \mathbf{H}$, $i_1\leq i\leq i_3$ and $C\in \mathcal{C}$. Let $s$ be the number of in-degree 0 or out-degree 0 vertices in $H_i$ and let $t$ be the number of vertices in $H_i$ with in- and out-degree 1.
For each $v\in [n]$, if $\mathcal{D}=\mathcal{H}$ and either $d^+_{H_i}(v)\leq d$ or $d^-_{H_i}(v)\leq d$, we have, as $D_{i_0}=D'_{i_0}=H_{i_0}$, that $v\in S_0$ and hence
$d^\diamond_{D_i}(v)=d^\diamond_{H_i}(v)$ for each $\diamond\in \{+,-\}$.
Therefore, if $\mathcal{D}=\mathcal{H}$, then $s_i=s$ and $t_i=t$.  Thus, given $\mathcal{D}=\mathcal{H}$, whether or not $C\in \mathcal{C}_{i}$ depends only on $\mathcal{H}$ and $i$, so we can assume that $C\in \mathcal{C}_{i}$, as otherwise \eqref{eq:clmeq} holds trivially. Similarly, if $d^+_{H_i}(v)+d^-_{H_i}(v)\leq 1$ for some $v\in [n]$, then  $d^+_{D_i}(v)+d^-_{D_i}(v)\leq 1$ and \eqref{eq:clmeq} again trivially holds. Thus, we can assume that $d^+_{H_i}(v)+d^-_{H_i}(v)\geq 2$ for each $v\in [n]$.

Given $\mathcal{D}=\mathcal{H}$, we have that $S=S_0$ is the fixed set $\{v\in [n]:d^+_{H_{i_0}}(v)\leq d\text{ or }d^-_{H_{i_0}}(v)\leq d\}$, which depends only on $\mathcal{H}$. Therefore, conditioned on $\mathcal{D}=\mathcal{H}$, the distribution of the random digraph $D_i$ is that of the deterministic graph $H_i$
with $i-e(H_i)$ edges uniformly at random added from $\{uv\notin E(H_i):u,v\in [n]\setminus S,u\neq v\}$. We will first replace these random edges with an appropriate binomial random digraph.

\myphead{Simplify the random edges within $[n]\setminus S$.}
Let $p=\log n/100n$, and let $\bar{D}$ be a binomial random digraph with edge probability $p$ and vertex set $[n]\setminus S$. Let $E'$ be the event that $e(\bar{D})\leq i-e(H_i)$. If $E'$ holds, then let $\widehat{D}$ be $\bar{D}$ with $i-e(H_i)-|E(\bar{D})\setminus E(H_i)|$ edges uniformly at random added
from $\{uv\notin E(H_i)\cup E(\bar{D}):u,v\in [n]\setminus S,u\neq v\}$.
If $E'$ does not hold, then let $\widehat{D}$ be the random digraph with vertex set $[n]\setminus S$ with $i-e(H_i)$ edges uniformly at random added from $\{uv\notin E(H_i):u,v\in [n]\setminus S,u\neq v\}$. Note that, $H_i\cup \widehat{D}$ has the same distribution as $D_i$ conditioned on $\mathcal{D}=\mathcal{H}$. Furthermore, if $E'$ holds, then $H_i\cup \bar{D}\subset H_i\cup \widehat{D}$.

As $\mathcal{H}\in \mathbf{H}$, we have that $\overline{\emph{\ref{rp6}}}$ holds whenever $\mathcal{D}=\mathcal{H}$, and therefore $i-e(H_i)\geq i_1-e(H_{i_1})\geq i_1-i_0-n\geq n\log n/40$, for sufficiently large $n$. Thus, as $\E(e(\bar{D}))\leq p(n-|S|)^2\leq pn^2$, by Lemma~\ref{chernoff}, we have $\P(E'\text{ does not hold})=\exp(-\Omega(n\log n))\leq \exp(-n)$, for sufficiently large $n$.

Thus, to complete the proof of the claim, it is sufficient to show that, for sufficiently large $n$,
$\P(C\not\subsetsim H_i\cup \bar{D})\leq \exp(-n)$,
for it follows that
\[
\P(C\not\subsetsim D_i|\mathcal{D}=\mathcal{H})=\P(C\not\subsetsim H_i\cup \widehat{D})\leq \P(E'\text{ does not hold})+\P(C\not\subsetsim H_i\cup \bar{D})\leq 2\cdot  \exp(-n).
\]
We now focus on the vertices in $H_i$ with low in- or out-degree.

\myphead{Identify low in- and out-degree vertices.}
Recalling that $d=\log n/300$, let
\[
Y=\{v\in [n]:d^+_{H_{i}}(v)\leq d\;\text{ or }\;d^-_{H_{i}}(v)\leq d\},
\]
so that $Y\subset S_1$.
For each $j\in \{0,1,2\}$, let $Y_j\subset Y$ be the vertices in $Y$ which, based on their out- and in-degree, could be a vertex of the copy of $C$ in $H_i$ with out-degree $j$. That is, let
\[
Y_0=\{v\in Y:d^-_{H_i}(v)\geq 2\},\;\; Y_1=\{v\in Y:d^\pm_{H_i}(v)\geq 1\},\;\; \text{ and }\;\; Y_2=\{v\in Y:d^+_{H_i}(v)\geq 2\}.
\]
As $d^-_{H_i}(v)+d^+_{H_i}(v)\geq 2$ for each $v\in [n]$, we have $Y=Y_0\cup Y_1\cup Y_2$. For each $\diamond\in \{+,-\}$, we will let $Y^\diamond\subset Y$ be the set of vertices with plenty of $\diamond$-neighbours in $H_i$, so that it is the set of vertices which are easy to place in the copy of $C$ as a vertex with two $\diamond$-neighbours. That is, we let, for each $\diamond\in \{+,-\}$,
\[
Y^\diamond=\{v\in Y:d^\diamond_{H_i}(v)> d\},
\]
so that, as $Y\subset S_1$, $Y^+$ and $Y^-$ are disjoint. Recall that $T$ is the set of vertices in $H_{i_1}\subset H_i$ with in- and out-degree both at most $d$, so that $Y=(Y\cap T)\cup Y^+\cup Y^-$.

Our aim now is to partition $Y$ into sets $\bar{Y}_j$, $j\in \{0,1,2\}$,  so that $|\bar{Y}_j|$, $j\in \{0,1,2\}$, satisfy certain inequalities for a later application of Lemma~\ref{pathincycle}. We will find them with vertices $x_v,y_v\in [n]\setminus Y$, for each $v\in Y$, where $x_v$ and $y_v$ can be the neighbours of $v\in \bar{Y}_j$ on any cycle in which $v$ has out degree $j$.
We do this differently according to the number of changes of direction on $C$. Let $\lambda$ then be the number of vertices of $C$ with out-degree 0.  As $C\in \mathcal{C}_{i}$, we have $1+(s_i-1)\log n\leq 2\lambda \leq n-1-(t_i-1)\log n$. Our cases are when $\lambda\geq n/4$ and when $\lambda<n/4$.

\myphead{Case I: when $\lambda\geq n/4$.} There are many changes of direction on $C$, so we will use vertices with out-degree 0 or 2 in $C$ to cover as many vertices as possible. Thus, let $\bar{Y}_0=(Y\cap Y^-)\cup (Y_0\cap T)$, $\bar{Y}_2=(Y\cap Y^+)\cup ((Y_2\cap T)\setminus \bar{Y}_0)$, and $\bar{Y}_1=Y\setminus (\bar{Y}_0\cup \bar{Y}_2)$. Note that each vertex in $\bar{Y}_1$ has in- and out-degree exactly 1 in $H_i$. Thus, $|\bar{Y}_1|= t_i$ and $\bar{Y}_1\subset T$.
As $2\lambda\leq n-1-(t_i-1)\log n$,  we have $n-2\lambda\geq 1+(t_i-1)\log n$, so that $\lceil (n-2\lambda)/\log n\rceil\geq t_i= |\bar{Y}_1|$. As $\overline{\emph{\ref{rp1}}}$ holds if $\mathcal{D}=\mathcal{H}$,
we have that $|S_0|\leq n^{2/3}$,
and thus
$|\bar{Y}_0|+|\bar{Y}_2|\leq |Y|\leq |S_0|\leq 2\lambda/\log n$.

Now, for each $v\in \bar{Y}_0 \cap T\subset Y_0$, pick distinct $x_v,y_v\in N^-_{H_i}(v)$. For each $v\in \bar{Y}_1\subset T$, pick $x_v\in N^+_{H_i}(v)$ and $y_v\in N^-_{H_i}(v)$. For each $v\in \bar{Y}_2\cap T\subset Y_2$, pick distinct $x_v,y_v\in N^+_{H_i}(v)$.
As $\mathcal{H}\in \mathbf{H}$, and $\overline{\emph{\ref{rp5a}}}$  holds whenever $\mathcal{D}=\mathcal{H}$, we have that $H_i[\{v,x_v,y_v:v\in Y\cap T\}]$ is a forest with $|Y\cap T|$ components and no vertices in $S_1\setminus (Y\cap T)$, and hence the vertices  $x_y,y_v$, $v\in Y\cap T$, are distinct and in $[n]\setminus S_1\subset [n]\setminus Y$.

As $\mathcal{H}\in \mathbf{H}$, and $\overline{\emph{\ref{rp3}}}$ holds whenever $\mathcal{D}=\mathcal{H}$, and as $H_i\subset K_3$ and $Y\subset S_1$, we have that each $v\in Y$ has at most 2 in- or out-neighbours in $H_i$
in $Y\cup N_{H_i}^+(Y-v)\cup N_{H_i}^-(Y-v)$. Therefore, for each $v\in \bar{Y}_0\setminus T\subset Y^-$, as $|N^-_{H_i}(v)|\geq d\geq 4$, we can do the following.
\stepcounter{propcounter}
\begin{enumerate}[label = \textbf{\Alph{propcounter}\arabic{enumi}}]
\item Pick distinct $x_v,y_v\in N^-_{H_i}(v)\setminus (Y\cup N^+(Y-v)\cup N^-(Y-v))$. \label{goodcase1}
\end{enumerate}
Similarly, for each $v\in \bar{Y}_2\setminus T\subset Y^+$, we can do the following.

\begin{enumerate}[label = \textbf{\Alph{propcounter}\arabic{enumi}}]\stepcounter{enumi}
\item Pick distinct $x_v,y_v\in N^+_{H_i}(v)\setminus (Y\cup N^+(Y-v)\cup N^-(Y-v))$. \label{goodcase2}
\end{enumerate}
Note that the vertices $x_v,y_v$, $v\in Y$, are distinct and in $[n]\setminus Y$.

To recap, we have found a partition $Y=\bar{Y}_0\cup \bar{Y}_1\cup \bar{Y}_2$, and distinct vertices $x_v,y_v$, $v\in Y$, in $[n]\setminus Y$, such that the following hold.
\stepcounter{propcounter}
\begin{enumerate}[label = \textbf{\Alph{propcounter}\arabic{enumi}}]
\item If $v\in \bar{Y}_0$, then $x_v,y_v\in N^-_{H_i}(v)$.
 If $v\in \bar{Y}_1$, then $x_v\in N^+_{H_i}(v)$ and $y_v\in N^-_{H_i}(v)$.
If $v\in \bar{Y}_2$, then $x_v,y_v\in N^+_{H_i}(v)$.\label{help1}
\item $|\bar{Y}_0|+|\bar{Y}_2|\leq \lceil 2\lambda/\log n\rceil$ and $|\bar{Y}_1|\leq \lceil (n-2\lambda)/\log n\rceil$.\label{help2}
\end{enumerate}

\myphead{Case II: when $\lambda<n/4$.} Supposing then that $\lambda<n/4$, we now find a partition $Y=\bar{Y}_0\cup \bar{Y}_1\cup \bar{Y}_2$, and distinct vertices $x_v,y_v$, $v\in Y$, in $[n]\setminus Y$, which also satisfy \ref{help1} and \ref{help2}.
This is more complicated than in Case I, as to achieve \ref{help2} we may have to assign vertices in $Y\setminus T$ not to $\bar{Y}_0\cup \bar{Y}_2$ but to $\bar{Y}_1$. Thus, all the vertices $x_v,y_v$, $v\in Y\setminus T$, are not selected in either \ref{goodcase1} or \ref{goodcase2}. To cope with this, we gather into a set $B$ the vertices $v\in Y\setminus T$ for which $x_v,y_v$ would be particularly hard to find were $v$ assigned to $\bar{Y}_1$. We then show there are enough changes of direction in $C$ to assign the vertices in $B$ instead to $\bar{Y}_0$ or $\bar{Y}_2$.

As $\overline{\emph{\ref{rp5a}}}$  holds whenever $\mathcal{D}=\mathcal{H}$, we can take $B\subset Y\setminus T$ to be a minimal set of vertices for which we have that $H_i[(Y\setminus B)\cup N_{H_i}^+(Y\setminus B)\cup N_{H_i}^-(Y\setminus B)]$ is a forest with $|Y\setminus B|$ components and no vertices in $B$. For each $\diamond\in \{+,-\}$, let $B^\diamond=\{v\in B:d^\diamond_{H_i}(v)> d\}$. Note that, as $B\subset Y\subset S_1$, $B^+$ and $B^-$ are disjoint and, as $B\cap T=\emptyset$, $B=B^+\cup B^-$.
Let $\bar{Y}_0=\{y\in Y:d^+_{H_i}(y)=0\}\cup B^-$. Let $\bar{Y}_2=\{y\in Y:d^-_{H_i}(y)=0\}\cup B^+$. Let $\bar{Y}_1=Y\setminus (\bar{Y}_1\cup \bar{Y}_2)\subset Y\setminus B$. As $\overline{\emph{\ref{rp1}}}$ holds if $\mathcal{D}=\mathcal{H}$,
we have that $|S_0|\leq n^{2/3}$, so that, as $\lambda<n/4$, we have $|\bar{Y}_1|\leq |Y|\leq |S_0|\leq n^{2/3}\leq (n-2\lambda)/\log n$.

If $i\geq i_2$, then $Y\subset S_2$ and, as $\overline{\emph{\ref{rp5b}}}$ holds whenever $\mathcal{D}=\mathcal{H}$, and $H_{i_2}\subset H_i$, we have that $B=\emptyset$. Hence, $|\bar{Y}_0|+|\bar{Y}_2|= s_i$, so that, as, $2\lambda\geq 1+(s_i-1)\log n$, we have $|\bar{Y}_0|+|\bar{Y}_2|= s_i\leq \lceil 2\lambda /\log n\rceil$.

If $i\leq i_2$, then, as $\overline{\emph{\ref{rp4}}}$ holds whenever $\mathcal{D}=\mathcal{H}$, and $H_i\subset H_{i_2}$, we have that
 $s_i\geq n^{1/5}$. As $\overline{\emph{\ref{rp2}}}$ and $\overline{\emph{\ref{rp2other}}}$ hold whenever $\mathcal{D}=\mathcal{H}$, we have that $|B|\leq 5n^{1/6}\leq s_i-2$. Thus, as $\lambda\geq 1+(s_i-1)\log n$, we have $|\bar{Y}_0|+|\bar{Y}_2|\leq s_i+|B|\leq 2s_i-2\leq 2\lambda /\log n$.

 Therefore, \ref{help2} holds whether $i\geq i_2$ or $i\leq i_2$.
Now, for each $v\in \bar{Y}_0\setminus B^-\subset Y_0$, pick distinct $x_v,y_v\in N^-_{H_i}(v)$. For each
 $v\in \bar{Y}_1\subset Y_1$, pick $x_v\in N^+_{H_i}(v)$ and $y_v\in N^-_{H_i}(v)$.
For each $v\in \bar{Y}_2\setminus B^+\subset Y_2$, pick distinct $x_v,y_v\in N^+_{H_i}(v)$. As $H_i[\{v,x_v,y_v:v\in Y\setminus B\}]\subset H_i[(Y\setminus B)\cup N_{H_i}^+(Y\setminus B)\cup N_{H_i}^-(Y\setminus B)]$ is a forest with $|Y\setminus B|$ components with no vertices in $B$, the vertices  $x_y,y_v$, $v\in Y\setminus B$, are distinct, and in $[n]\setminus Y$.

Similarly to in Case I, we have that each $v\in Y$ has at most 2 in- or out-neighbours in $Y\cup N_{H_i}^+(Y-v)\cup N_{H_i}^-(Y-v)$.
 Therefore, for each $v\in B^-$, as $|N^-_{H_i}(v)|\geq d\geq 4$, we can pick distinct $x_v,y_v\in N^-_{H_i}(v)\setminus (Y\cup N_{H_i}^+(Y-v)\cup N_{H_i}^-(Y-v))$.
Similarly, for each $v\in B^+$, we can pick distinct $x_v,y_v\in N^+_{H_i}(v)\setminus (Y\cup N_{H_i}^+(Y-v)\cup N_{H_i}^-(Y-v))$.
Note that all the vertices $x_v,y_v$, $v\in Y$, are distinct and in $[n]\setminus Y$, and therefore \ref{help1} holds.

\medskip

We have now chosen, in both Case I and Case II, a partition $Y=\bar{Y}_0\cup \bar{Y}_1\cup \bar{Y}_2$, and distinct vertices $x_v,y_v$, $v\in Y$, in $[n]\setminus Y$, satisfying \ref{help1} and \ref{help2}.
We now choose the section of the cycle, and vertex sets in it, that we use to cover the low in- and out-degree vertices (those in $Y=\bar{Y}_0\cup \bar{Y}_1\cup \bar{Y}_2$).

\myphead{Choosing $P$ and $Z_0,Z_1,Z_2$.} Now, by Lemma~\ref{pathincycle} and \ref{help2}, there exists a path $P\subset C$ with length at most $n/100$ and vertex sets $Z_0,Z_1,Z_2\subset V(P)$ such that the following hold with $Z=Z_0\cup Z_1\cup Z_2$.
\stepcounter{propcounter}
\begin{enumerate}[label = \textbf{\Alph{propcounter}\arabic{enumi}}]
\item For each $j\in \{0,1,2\}$, $|Z_j|=|\bar{Y}_j|$, and the vertices in $Z_j$ each have out-degree $j$ in $C$.\label{RR}
\item The vertices in $Z$ are pairwise at least $100\log n/\log\log n$ apart on $P$ from each other.\label{basedonatruestory}
\end{enumerate}

Using \ref{RR}, label the vertices of $Z$ as $a_v$, $v\in Y$, so that, for each $j\in \{0,1,2\}$, if $v\in \bar{Y}_j$, then $a_v\in Z_j$. Pick an arbitrary direction on $C$ to be clockwise and, for each $v\in Y$, label vertices of $C$ so that $d_vb_va_vc_ve_v$ is a path on $C$ with vertices in clockwise order. Using \ref{help1} and \ref{RR}, for each $v\in Y$, by swapping the labels of $x_v$ and $y_v$ if necessary, we can assume that $x_vvy_v$ is a copy of $b_va_vc_v$.

We now modify $H_i$, $C$ (along with $P$) and $\bar{D}$, to allow us to apply Theorem~\ref{mainthm}.

\myphead{Modify $H_i$ to get a pseudorandom graph, $H_i'$.}
For each $v\in Y$, we wish to delete $v,x_v$ and $y_v$ from $H_i$ and replace them with a single new vertex $z_v$. Later we will find a cycle $C_0$ containing a subpath through $z_v$, say with vertices $w,z_v,w'$, and then replace $wz_vw'$ with $wx_vvy_vw'$ to get a copy of the subpath $d_vb_va_vc_ve_v$ of $P$. To do this, we use the in-edges of $z_v$ to guarantee the appropriate edge between $w$ and $x_v$ (matching the edge between $d_v$ and $b_v$) and the out-edges of $z_v$ to guarantee the appropriate edge between $y_v$ and $w'$ (matching the edge between $c_v$ and $e_v$), and insist later that the subpath on $w,z_v,w'$ in $C_0$ is directed from $w$ to $w'$.

More precisely, define $H_i'$ as follows. Let $Y':=\{v,x_v,y_v:v\in Y\}$ and let $H_i'$ be the graph formed by deleting $Y'$
from $H_i$ and adding the new vertices $z_v$, $v\in Y$, along with the following additional edges for each $v\in Y$ and $w\in [n]\setminus Y'$.

\bigskip

\begin{minipage}{0.475\textwidth}
\stepcounter{propcounter}
\begin{enumerate}[label = \textbf{\Alph{propcounter}\arabic{enumi}}]
  \item ${wz_v}$ if $d_vb_v\in E(C)$ and $wx_v\in E(H_i)$\label{edge1}
  \item ${wz_v}$ if $b_vd_v\in E(C)$ and $x_vw\in E(H_i)$\label{edge2}
  \end{enumerate}
  \end{minipage}
  \begin{minipage}{0.475\textwidth}
\begin{enumerate}[label = \textbf{\Alph{propcounter}\arabic{enumi}}]\addtocounter{enumi}{2}
      \item ${z_vw}$ if $c_ve_v\in E(C)$ and $y_vw\in E(H_i)$\label{edge3}
            \item ${z_vw}$ if $e_vc_v\in E(C)$ and $wy_v\in E(H_i)$\label{edge4}
\end{enumerate}
\end{minipage}

\bigskip

Let $X=\{z_v:v\in Y\}$. Note that $|X|= |Y|\leq n^{2/3}$. Let $\bar{n}=|H_i'|=n-2|X|\geq n-2n^{2/3}$, so that, for large $n$, $|X|\leq \bar{n}^{3/4}$. We will show that $H_i'$ is pseudorandom with exceptional set $X$. As $\overline{\emph{\ref{rpnew}}}$ holds whenever $\mathcal{D}=\mathcal{H}$,
and $H_i\subset K_{i_3}$, we have that $\Delta(H_i')\leq \Delta(H)\leq 50n\log n\leq 100\bar{n}\log \bar{n}$, and therefore \ref{pseudo0} holds for~$D=H_i'$.

As $\overline{\emph{\ref{rp3}}}$  holds whenever $\mathcal{D}=\mathcal{H}$, we have that any vertex in $[n]$ has at most 2 in- or out-neighbours in $H_i$ in $S_1\cup N_{H_i}^+(S_1)\cup N_{H_i}^-(S_1)$, and hence at most 2 in- or out-neighbours in $H_i$ in $Y'$. Thus, for each $v\in [n]\setminus Y'\subset [n]\setminus Y$ and $\diamond\in \{+,-\}$, we have
\begin{equation*}
d_{H_i'}^\diamond(v,V(H_i')\setminus X)= d_{H_i}^\diamond(v,[n]\setminus Y')\geq d-2\geq \log n/500\geq \log \bar{n}/500.
\end{equation*}
For each $v\in Y$, we have that if $d_vb_v\in E(C)$, then $d_{H_i}^-(x_v,[n]\setminus Y')$ in-edges to $z_v$ are added at \ref{edge1}, and if $b_vd_v\in E(C)$, then $d_{H_i}^+(x_v,[n]\setminus Y')$ in-edges to $z_v$ are added at \ref{edge2}.
Therefore, as $x_v\notin Y$,
\[
d_{H_i'}^-(z_v,V(H_i')\setminus X)\geq \min\{d_{H_i}^-(x_v,[n]\setminus Y'),d_{H_i}^+(x_v,[n]\setminus Y')\} \geq d-2\geq \log n/500\geq \log \bar{n}/500.
\]
Similarly, enough out-edges from $z_v$
are added at either \ref{edge3} or \ref{edge4} that $d_{H_i'}^+(z_v,V(H_i')\setminus X)\geq \log \bar{n}/500$. Therefore, \ref{pseudo1} holds for $D=H_i'$ and $X$.

We now prove that \ref{pseudo2} holds for $D=H_i'$. Suppose, for contradiction, that there are sets $A',B'\subset V(H_i')$ and some $\diamond\in \{+,-\}$ with $d^\diamond_{H_i'}(v,B')\geq (\log \bar{n})^{2/3}$ for each $v\in A'$, $|B'|\leq |A'|(\log \bar{n})^{1/3}$ and $|A'|\leq \bar{n}\log\log \bar{n}/\log \bar{n}$.
Now, every vertex in $v\in [n]\setminus Y'$ has at most 2 in- or out-neighbours in $H_i$ in $Y'$, and hence at most 2 in- or out-neighbours in $H'_i$ in $X$. Therefore, every vertex in $A'\setminus X$ has at least $(\log \bar{n})^{2/3}-2\geq (\log n)^{2/3}/2$ $\diamond$-neighbours in $B'\setminus X$. Therefore, as $\overline{\emph{\ref{rp8}}}$ holds whenever $\mathcal{D}=\mathcal{H}$, and any edges in $H_i'$ between $A'\setminus X$ and $B'\setminus X$ lie in $H_i\subset K_3$, we have $|A'\setminus X|\leq |B'|/100(\log n)^{1/3}\leq |A'|/2$. Thus, $|A'\cap X|\geq |A'|/2$.

Now, take $A''\subset A'\cap X$ with $|A''|\geq |A'|/4$ so that all the $\diamond$-edges of $z_v\in A''$ were added under the same step \ref{edge1}--\ref{edge4} (noting that only two steps are used for each possible value of $\diamond$).  If these edges were all added under \ref{edge1}, then let $A_0=\{x_v:z_v\in A''\}$, and observe that every vertex in $A_0$ has at least $(\log \bar{n})^{2/3}-2\geq (\log n)^{2/3}/2$ in-neighbours in $H_i$ in $B'\setminus X$. Therefore, again as $\overline{\emph{\ref{rp8}}}$ holds whenever $\mathcal{D}=\mathcal{H}$, we have $|B'|\geq |B'\setminus X|\geq 100|A'\cap X|(\log n)^{1/3}> |A'|(\log n)^{1/3}$, a contradiction. A similar contradiction is reached if all these edges are added at \ref{edge2}, at \ref{edge3}, or at \ref{edge4}.
Thus, no such sets $A'$ and $B'$ exist, so that \ref{pseudo2} holds for $D=H_{i}'$. Therefore, $H_i'$ is a pseudorandom digraph with exceptional set $X$.

\myphead{Modify $C$ and $P$.} For each $v\in Y$, recall that the labelled vertices $d_v,b_v,a_v,c_v,e_v$ occur consecutively in this order on $C$. Note that, by \ref{basedonatruestory}, all these labelled vertices are distinct. Let $C'$ be the cycle formed by, for each $v\in Y$, deleting the vertices $a_v,b_v,c_v$ and adding the new vertex $f_v$ along with the edges $d_vf_v$ and $f_ve_v$. Note these new edges are the same whether $d_vb_v$ or $b_vd_v$ is an edge of $C$, and whether $c_ve_v$ or $e_vc_v$ is an edge of $C$. Let $P'$ be the path $P$ with these same modifications carried out, so that $P'\subset C'$.

\myphead{Modify $\bar{D}$.} We modify $\bar{D}$ similarly to our modification for $H_i$. For each $v\in Y$, we wish to delete $v,x_v$ and $y_v$ from $\bar{D}$ and replace them with $z_v$ (as created for $H_i'$), and add an in-edge from $w\in [n]\setminus Y'$ to $z_v$ if there is an appropriate edge between $w$ and $x_v$ in $\bar{D}$ (matching the edge between $d_v$ and $b_v$) and an out-edge if there is an appropriate edge between $y_v$ and $w$ in $\bar{D}$ (matching the edge between $c_v$ and $e_v$). This will result in there being an edge from $w$ to $x_v$ with probability $p$, independently of all other edges in $\bar{D}$, and, similarly, an edge from $x_v$ to $w$ independently at random with probability $p$.

More precisely, let $\bar{D}'$ be the random graph formed by deleting $Y'=\{v,x_v,y_v:v\in Y\}$ from $\bar{D}$, and adding the vertices $z_v$, $v\in Y$, and the following edges for each $v\in Y$ and $w\in [n]\setminus Y'$.

\smallskip

\begin{minipage}{0.475\textwidth}
\stepcounter{propcounter}
\begin{enumerate}[label = \textbf{\Alph{propcounter}\arabic{enumi}}]
\item ${wz_v}$ if $d_vb_v\in E(C)$ and $wx_v\in E(\bar{D})$\label{rand1}
\item ${wz_v}$ if $b_vd_v\in E(C)$ and $x_vw\in E(\bar{D})$\label{rand2}
  \end{enumerate}
  \end{minipage}
  \begin{minipage}{0.475\textwidth}
  \begin{enumerate}[label = \textbf{\Alph{propcounter}\arabic{enumi}}]\addtocounter{enumi}{2}
      \item ${z_vw}$ if $c_ve_v\in E(C)$ and $y_vw\in E(\bar{D})$\label{rand3}
    \item ${z_vw}$ if $e_vc_v\in E(C)$ and $wy_v\in E(\bar{D})$\label{rand4}
\end{enumerate}
\end{minipage}
\medskip

Finally, for each distinct $v,v'\in Y$, add the edge $z_vz_{v'}$ independently at random with probability $p$. Note that the distribution of $\bar{D}'$ is the same as the binomial random digraph with vertex set $V(H_i')$ and edge probability $p$.

\myphead{Apply Theorem~\ref{mainthm}.} Define $f:X\to V(P')$ by letting $f(z_v)=f_v$ for each $v\in Y$. Suppose $H'_{i}\cup \bar{D}'$ contains some $C_0$ which is a copy  of $C'$ in which $f_v$ is copied to $z_v$ for each $v\in Y$. We will show that, then, $H_i\cup \bar{D}$ contains a copy of $C$.

For each $v\in Y$, let $\alpha_v$ and $\beta_v$ be the copies of $d_v$ and $e_v$ on $C_0$, respectively. Let $P_v$ be the path $C[\{d_v,b_v,a_v,c_v,e_v\}]$. We will show that $\phi_v:P_v\to H_i\cup \bar{D}$ defined by $\phi_v(d_v)=\alpha_v$, $\phi_v(b_v)=x_v$, $\phi_v(a_v)=v$, $\phi_v(c_v)=y_v$ and $\phi_v(e_v)=\beta_v$ is an embedding of $P_v$ into $H_i\cup \bar{D}$, for each $v\in Y$.

By the choice of $C'$, we have $d_vf_v,f_ve_v\in E(C')$, and hence $\alpha_vz_v,z_v\beta_v\in E(H_i'\cup \bar{D}')$.
If $d_vb_v\in E(C)$ and $\alpha_vz_v\in E(H_i')$, then, as $\alpha_vz_v$ was added to $H_{i}'$ at \ref{edge1}, we have $\alpha_vx_v\in E(H_i)$. If $d_vb_v\in E(C)$ and $\alpha_vz_v\in E(\bar{D}')$, then, by the choice of edges at \ref{rand1}, we have $\alpha_vx_v\in E(\bar{D})$. Thus, if $d_vb_v\in E(C)$, then $\alpha_vx_v\in E(H_i)\cup E(\bar{D})$.
Similarly, considering \ref{edge2} and \ref{rand2}, if $b_vd_v\in E(C)$, then $x_v\alpha_v\in E(H_i)\cup E(\bar{D})$. Therefore, $\phi_v$ restricted to $\{d_v,b_v\}$ is an embedding of $P_v[\{d_v,b_v\}]$ into $H_i\cup \bar{D}$.
Similarly, from \ref{edge3}, \ref{edge4}, \ref{rand3}, and \ref{rand4}, we have that $\phi_v$ restricted to $\{c_v,e_v\}$ is an embedding of $P_v[\{c_v,e_v\}]$ into $H_i\cup \bar{D}$.
Finally, by the labelling after \ref{RR}--\ref{basedonatruestory}, we have that $\phi_v$ restricted to $\{b_v,a_v,c_v\}$ is an embedding of $P_v[\{b_v,a_v,c_v\}]$ into $H_i\cup \bar{D}$.
Thus, $\phi_v:P_v\to H_i\cup \bar{D}$ is an embedding of $P_v$ into $H_i\cup \bar{D}$, for each $v\in Y$. As $C'$ was formed by replacing $C[d_v,b_v,a_v,c_v,e_v]$ by $d_f\to f_v\to e_v$, for each $v\in Y$, if we take  each path $\alpha_vz_v\beta_v$, $v\in Y$, on $C_0$, and replace it with $\phi(P_v)$, we get a copy of $C$ in $H_i\cup\bar{D}$.

Finally, by Theorem~\ref{mainthm}, with probability at least $1-2\exp(-2\bar{n})$, $H_i'\cup \bar{D}'$ contains a copy of $C'$ in which $f_v$ is copied to $z_v$ for each $v\in Y$. As, whenever this happens, $H_i\cup \bar{D}$ contains a copy of $C$, we have
\[
\P(C\subsetsim H_i\cup \bar{D})\geq 1-\exp(-2\bar{n})\geq 1-\exp(-n).
\]
This finishes the proof of the claim, and thus the theorem.
\renewcommand{\qedsymbol}{$\boxdot$}
\qedhere
\renewcommand{\qedsymbol}{$\square$}
\qedsymbol
\end{proof}
\renewcommand{\qedsymbol}{}
\end{proof}
\renewcommand{\qedsymbol}{$\square$}
\fi

\subsection{Proof of Theorem~\ref{hittime}}\label{sec:imp1}
We now show Theorem~\ref{thmrandprocess} implies Theorem~\ref{hittime}, which we restate for convenience.
\hittime*
\begin{proof}[Proof of Theorem~\ref{hittime} from Theorem~\ref{thmrandprocess}]
Note first that, by  \emph{\ref{rp4}} in Lemma~\ref{lem:keyrandprops}, we have, with high probability, that $m_1\geq i_2=3n\log n/4$. Furthermore, $D_{m_1+1}$ has no vertices with in- or out-degree 0, while $D_{m_1}$ has exactly one such vertex.
By Theorem~\ref{thmrandprocess}, then, for both  \emph{(i)} and \emph{(ii)} of Theorem~\ref{hittime} it is sufficient to show that, with high probability, there are no vertices $v\in V(D_{i_2})$ with $d^+_{D_{i_2}}(v),d^-_{D_{i_2}}(v)\leq 1$.

Let $p_1=(i_2-n)/n(n-1)$ and $D=D(n,p_1)$. Similarly as for \eqref{sizecalc}, by Lemma~\ref{chernoff}, we have that $\P(e(D)\leq i_2)=1-o(1)$. Therefore, it is sufficient to show that, with high probability, $D$ has no vertices $v\in V(D)$ with $d^+_{D}(v),d^-_{D}(v)\leq 1$. The probability that such a vertex does exist is at most
\[
n\sum_{k=0}^2\binom{2n-2}{k}p_1^k(1-p_1)^{2(n-1)-k}\leq 3n(2p_1n)^2\exp(-p_1(2n-4))=o(1),
\]
as required.

For \emph{(iii)} in Theorem~\ref{hittime}, by Theorem~\ref{thmrandprocess} it is sufficient to show that, with high probability, $D_{m_0}$ has at most $n^{1/2}\log^2n$ vertices with in-degree 0 or out-degree 0, and at most $\log^5n$ vertices $v\in D_{m_0}$ with $d^+_{D_{m_0}}(v)=d^-_{D_{m_0}}(v)=1$.

From \emph{\ref{rp7other}} in Lemma~\ref{lem:keyrandprops2}, we have, with high probability, that $m_0\geq i_1=(n\log n)/2-n\log\log n$.
Let $p_2=(i_1-n)/n(n-1)$ and $\widehat{D}=D(n,p_2)$. Similarly as for \eqref{sizecalc}, by Lemma~\ref{chernoff}, we have that $\P(e(\widehat{D})\leq i_1)=1-o(1)$. Therefore, it is sufficient to show that, with high probability, $D_{2}$ has at most $n^{1/2}\log^2n$ vertices with in-degree 0 or out-degree 0, and at most $\log^5n$ vertices $v\in V(D_{2})$ with $d^+_{\widehat{D}}(v),d^-_{\widehat{D}}(v)\leq 1$.

Let $X_1$ be the number of vertices with in-degree 0 or out-degree 0 in $\widehat{D}$. Then,
\[
\E X_1\leq 2n(1-p_2)^{n-1}\leq 2n\exp(-p_2(n-1)) =O(n^{1/2}\log n).
\]
Therefore, with high probability, $X_1\leq n^{1/2}\log^2 n$.

Let $X_2$ be the number of vertices $v\in V(\widehat{D})$ with $d^+_{\widehat{D}}(v),d^-_{\widehat{D}}(v)\leq 1$. Then,
\[
\E X_2\leq n\sum_{k=0}^2\binom{2n-2}{k}p_1^k(1-p_1)^{2(n-1)-k}\leq 3n(2p_1n)^2\exp(-p_2(2n-4))=O(\log^4n).
\]
Therefore, with high probability, $X_2\leq \log^5 n$. This completes the proof of \emph{(iii)} in Theorem~\ref{hittime}.
\end{proof}

\subsection{Proof of Theorem~\ref{sharpthres}}\label{sec:imp2}
Finally, using Theorem~\ref{thmrandprocess}, we deduce Theorem~\ref{sharpthres}, which we also restate for convenience.
\sharpthres*
\begin{proof}[Proof of Theorem~\ref{sharpthres} from Theorem~\ref{thmrandprocess}] Let $\mathcal{C}_n$ be the set of all $n$-vertex oriented cycles whose underlying undirected cycle is the canonical cycle with vertex set $[n]$. Recall that, for each $C\in \mathcal{C}_n$, $\lambda(C)$ is the number of vertices of $C$ with in- or out-degree 0, and $p_C=\max\{\log n,2(\log n-\lambda(C))\}/2n$ if $\lambda(C)>0$, and $p_C=\log n/n$ otherwise.

Let $\eps>0$ be small and fixed and $p=p(n)$. Note that, for Theorem~\ref{sharpthres}, we can assume that $(\min_{C\in \mathcal{C}_n}p_C)/(1-\eps)\leq p\leq (\max_{C\in \mathcal{C}_n}p_C)/(1+\eps)$, and thus that $\log n/2(1-\eps)n\leq p\leq \log n/(1+\eps)n$.

Now,  $(1+\eps)p\geq (1+\eps)\log n/2(1-\eps)n>\log n/2n$, so if we have $C\in \mathcal{C}_n$ and $p_C\geq (1+\eps)p$, then
\[
\frac{\log n-\log \lambda(C)}{n}=p_C\geq (1+\eps)p,
\]
and hence $C$ has at most $n\exp(-(1+\eps)pn)$ vertices with in-degree 0.
 On the other hand, we will show that $D=D(n,p)$ is likely to have more than $n\exp(-(1+\eps)pn)$ vertices with in-degree 0, and thus contain no such cycle.
 Note that, for each $v\in V(D)$, the probability that $d^-_{D}(v)=0$ is $(1-p)^{n-1}\geq \exp(-(1+\eps/2)pn)$. Furthermore, this is independent for each $v\in V(D)$. Thus, as $(1+\eps)pn\leq \log n$ and the expected number of vertices with $d^-_{D}(v)=0$ is at least $n\cdot \exp(-(1+\eps/2)pn)$, by Lemma~\ref{chernoff}, with high probability there are more than $n\exp(-(1+\eps)pn)$ vertices with degree $0$ in $D$. Thus, with high probability, $D(n,p)$ contains no cycle $C\in \mathcal{C}_n$ with $p_C\geq (1+\eps)p$.

Now, if $C\in \mathcal{C}_n$ has $p_C\leq (1-\eps)p$, then
\[
\frac{\log n-\log \lambda(C)}{n}\leq (1-\eps)p,
\]
and hence $C$ has at least $n\exp(-(1-\eps)pn)$ vertices with in-degree 0.
On the other hand, the expected number of vertices in $D(n,p)$ with out-degree 0 or in-degree 0 is at most
\[
2n(1-p)^{n-1}\leq 2n\exp(-p(n-1)) =o(n\exp(-(1-\eps)pn)/\log^2n).
\]
Therefore, with high probability, $D(n,p)$ has at most $n\exp(-(1-\eps)pn)/\log^2n$ vertices with in- or out-degree 0.
Furthermore, as $\log n/2(1-\eps)n\leq p\leq \log n/(1+\eps)n$, the probability that  $D(n,p)$ contains a vertex with total in- and out-degree less than 3 is at most
\[
n\cdot \sum_{i=0}^2\binom{2n-2}{i}p^i(1-p)^{2(n-1)-i}\leq 3n(2np)^2\exp(-p(2n-4)) =o(1).
\]
Therefore, with high probability, each vertex in $D(n,p)$ has total in- and out-degree at least 3 (and hence, in particular, no vertices with in- and out-degree both 1).

Let $\mathcal{P}$ be the property of digraphs $D$ such that, for all $s, t$ and $n$, if $D$ has $n$ vertices, $s$ of which have in-degree 0 or out-degree 0 and $t$ of which have in-degree 1 and out-degree 1, and $d^+_D(v)+d^-_D(v)\geq 2$, then $D$ contains a copy of every $n$-vertex cycle with at least $1+(s-1)\log n$ changes in direction and at most $n-1-(t-1)\log n$ changes in direction. Thus, it is sufficient to complete the proof of the theorem to show that, with high probability, $D(n,p)\in \cP$.

Let $\eta>0$. By Theorem~\ref{thmrandprocess}, there is some $n_0$ such that, for each $0\leq M\leq n(n-1)$, if $D_{n,M}$ is a random digraph chosen uniformly from those with vertex set $[n]$ and $M$ edges, then $\P(D_{n,M}\in \mathcal{P})\geq 1-\eta$.
Then,
\[
\P(D(n,p)\in \mathcal{P})=\sum_{M=0}^{n(n-1)}\P(e(D(n,p))=M)\cdot \P(D_{n,M}\in \mathcal{P})\geq \sum_{M=0}^{n(n-1)}\P(e(D(n,p))=M)\cdot (1-\eta)=(1-\eta).
\]
Thus, as $n_0$ was chosen depending only on $\eta$, we have, with high probability, that $D(n,p)$ is in $\mathcal{P}$, as required.
\end{proof}

\section*{Acknowledgements}
The author would like to thank Asaf Ferber and Benny Sudakov for useful conversations behind this paper.


\end{document}